\documentclass{amsart}
\usepackage{amssymb}
\usepackage{verbatim}
\usepackage{array}
\usepackage{latexsym}
\usepackage{enumerate}
\usepackage{enumitem}
\usepackage{amsmath,mathtools}
\usepackage{mathrsfs}
\usepackage{amsfonts}
\usepackage{amsthm}
\usepackage[foot]{amsaddr}
\usepackage[table,svgnames,dvipsnames]{xcolor}
\usepackage{caption}
\usepackage[english]{babel}
\usepackage{graphicx}
\usepackage{tikz}
\usepackage{tikz-cd}
\usepackage{adjustbox}

\newcommand{\mycomment}[1]{}

\usepackage[autostyle]{csquotes}

\usepackage{hyperref}
\usepackage[capitalise]{cleveref}

\newtheorem{theorem}{Theorem}[section]

\newtheorem{lemma}[theorem]{Lemma}
\newtheorem{corollary}[theorem]{Corollary}
\newtheorem{definition}[theorem]{Definition}

\newtheorem{example}[theorem]{Example}

\newtheorem{remark}[theorem]{Remark}

\def\A{\mathbb A}
\newcommand{\Fq}{{\mathbb{F}_q}}

\newcommand{\F}{\mathbb{F}}
\newcommand{\Z}{\mathbb{Z}}

\renewcommand{\P}{\mathbb{P}}

\numberwithin{equation}{section}

\newcommand{\T}{\mbox{T}}

\newcommand{\calS}{\mathcal{S}}
\newcommand{\calZ}{\mathcal{Z}}
\newcommand{\calY}{\mathcal{Y}}

\newcommand{\Zgamma}{\calZ_\gamma}

\newcommand{\ZZ}{\overline{\calZ}_{\gamma}}

\DeclareMathOperator{\supp}{Supp}

\title[Hierarchical Locally Recoverable Codes on surfaces]{Hierarchical Locally Recoverable Codes on surfaces}
\date{}

\author[C. Araujo]{Carolina Araujo$^1$} \address{$^1$IMPA, Estrada Dona Castorina 110, 22460-320 Rio de Janeiro, Brazil} \email{caraujo@impa.br} \thanks{}
\author[L. Costa]{Luana Costa$^1$} \address{} \email{luana.costa@impa.br} \thanks{}
\author[B. Malmskog]{Beth Malmskog$^2$} \address{$^2$Colorado College Department of Mathematics and Computer Science, Colorado Springs, Colorado USA} \email{bmalmskog@colroadocollege.edu} \thanks{}
\author[J. Mello]{Jorge Mello$^3$} \address{$^3$Department of Mathematics and Statistics, Oakland University, Michigan, USA} \email{jorgedemellojr@oakland.edu} \thanks{}
\author[E. Menezes]{Eliza Menezes$^1$} \address{} \email{elizabeth.menezes@impa.br} \thanks{}
\author[C. Salgado]{Cecília Salgado$^4$} \address{$^4$Faculty of Science and Engineering - Bernoulli Institute, University of Groningen, Groningen 9747AG, The Netherlands} \email{c.salgado@rug.nl} \thanks{}
\author[L. Vicino]{Lara Vicino$^4$} \address{} \email{l.vicino@rug.nl} \thanks{}

\begin{document}

\begin{abstract}
We construct locally recoverable codes with hierarchy from surfaces in $\mathbb{A}^3$ admitting a fibration by curves of Artin-Schreier or Kummer type. We derive the parameters of our codes by leveraging the geometry and arithmetic of the fibration, which is obtained by projection onto one of the coordinates. As a byproduct, we obtain estimates for (and in one case an explicit count of) the number of rational points in certain families of surfaces.
\end{abstract}

\maketitle

\thanks{{\em Keywords}}: Locally Recoverable Codes, hierarchy, fibered surfaces, Artin-Schreier curves, Kummer curves

\thanks{{\em MSC codes}}: 11G20, 14G50, 94B27

\section{Introduction}

Error-correcting codes provide a method of incorporating thoughtful redundancy into information that is to be stored or transmitted over a potentially noisy channel, so that errors can be corrected or missing information recovered.  Different uses require codes with different strengths.  Cloud storage applications pose a unique challenge in that entire servers routinely fail or become unavailable, so information stored entirely on a single server is at risk of disappearing completely, no matter how well redundancy is built into individual codewords.  Simply backing up each server at least doubles storage overhead, making it desirable to find a more efficient way to spread each piece of information across many servers so that any missing data can be recovered from other servers.   \textit{Locally recoverable codes} (LRCs) are error-correcting codes with the property that, if any symbol of a codeword is erased, there is some other relatively small collection of code symbols that can be used to recover the missing value.  This collection of symbols is known as a \textit{recovery set} for the erased position. LRCs were introduced in \cite{PKLK12} and have been widely studied in recent years.  An LRC with \textit{hierarchy} (sometimes abbreviated as an HLRC) has a second, larger recovery set for each position, which may allow the recovery of additional erasures.  HLRCs, which were first introduced in \cite{sasidharan2015codes}, are one way to create a ``back-up plan" for local recovery, so that the erasure of a few positions can be repaired without needing to access most of the symbols in the codeword.

The rich geometric and algebraic structures of curves and higher dimensional varieties over finite fields can naturally give rise to evaluation codes on algebraic varieties with structures as locality and hierarchy. Algebraic curves have been key components in several interesting constructions of LRCs (see e.g.~\cite{tamobarg2014, barg2017locally, tafazoliantop2025}), HLRCs (see e.g.~\cite{ballentine2019codes} and \cite{haymaker2025algebraic}) and LRCs with availability (e.g.~in \cite{haymaker2018locally}). While codes on surfaces have been studied extensively, including some work on LRCs from surfaces (see e.g.~\cite{barg2017surfaces,SVAV2021,aguilar25,SV25}), the potential of surfaces, and in particular of surfaces admitting a fibration by curves with many $\Fq$-points, to naturally yield HLRCs has only recently emerged in the literature, pioneered in \cite{berg2024codes}. 
While products of curves provide natural examples of surfaces admitting fibrations, considering more general fibrations increases the range of possibilities and provides many more examples.

Our point of departure are affine surfaces of the form $\mathcal{S}:h(y)=f(x,z) \subset \mathbb{A}^3_{(x,y,z)}$ such that specialization in one of the variables, say $z$ for illustration purposes, yields a family of curves $\mathcal{Z}_{\gamma}: h(y)=f(x,\gamma)$ with many $\Fq$-points.
The evaluation points of the code is given by a carefully chosen subset $T$ of the set of $\Fq$-rational points of $S$. 
We then consider a vector space $V$ of polynomials $g(x,y,z)$, with suitable bounds on the degrees in $x$, $y$, and $z$, and the evaluation map $V\rightarrow \mathbb{F}_q^n$, where $n=|T|$. 
While the choice of the evaluation set and the construction of the vector space are naturally done in the affine setting, we use a (partial) compactification of $\mathcal{S}$ in $\mathbb{P}^2_{(x:y:w)}\times \mathbb{A}^1_{z}$ in order to apply B\'ezout's theorem on each curve $\calZ_\gamma$ and obtain estimates for the minimal distance.

More concretely, we consider an equation for $\mathcal{S}$ of the form \begin{itemize}
\item $y^p-y=f(x,z)$, where $p$ is prime and $q$ is a power of $p$, and $f(x,z)$ is a polynomial, or
\item $y^m=f(x,z)$, where $\gcd(q,m)=1$, $\Fq$ contains a primitive $m-$th root of unity, and $f(x,z)$ is a polynomial.
\end{itemize}
In the first case, the curve $\Zgamma$ defined by $y^p-y=f(x,\gamma)$ for some $\gamma\in\Fq$ is generically an Artin-Schreier curve, the projective closure of which is a degree-$p$ Galois cover of the projective line.  In the second case, the curve $\Zgamma$ defined by $y^m=f(x,\gamma)$ is generically a Kummer curve, the projective closure of which is a cyclic degree-$m$ Galois cover of the projective line.  
These are excellent curve families for our application because the structure of the curves in these families allows us to reasonably count their $\Fq$-rational points.

In this article, we treat each of these two classes of surfaces in detail and, for each of them, work out a concrete example with an explicit equation (see \cref{sec:ASexample} and \cref{sec:KummerEx}).
We observe that these two explicit examples exhibit special geometric features. In the special Artin-Schreier construction discussed in \cref{sec:ASexample},  the projectivization of the corresponding surface in $\P^3_{(x:y:z:w)}$ contains lines at infinity, namely $\{x=w=0\}$ and $\{z=w=0\}$.
The fibration described above, namely, the specialization in the variable $z$, corresponds to the rational map $\P^3\dashrightarrow \P^1$ given by projection from the line $\{z=w=0\}$.
One could also project from the other line at infinity, $\{x=w=0\}$, to obtain a different fibration by Artin-Schreier curves on the same surface. In the second example (\cref{sec:KummerEx}), the surface is the affine cone over a curve that is covered by the Hermitian curve, and its equation is symmetric in $x$ and $z$. Specialization in the variable $x$ gives another fibration by Kummer curves on the same surface.

\subsection{Outline of the article}

In \cref{background}, we introduce some background on evaluation codes and HLRCs, together with fundamental concepts from algebraic geometry and from function field theory that we need throughout the paper.
In \cref{sec:AS}, we introduce a general construction of HLRCs on surfaces with Artin–Schreier fibers and explicitly bound their parameters. In \cref{sec:ASexample}, we discuss, in our framework, an explicit example from \cite{berg2024codes}. Similarly, in \cref{sec:Kummer}, we present a general construction of HLRCs on surfaces with Kummer fibers, and in \cref{sec:KummerEx} we study an explicit example.
This work is of independent interest as we count points on a new family of surfaces using interesting geometric arguments.

\section{Notations and background}\label{background}

Throughout this paper, let $p$ be a prime number, $q=p^h$, for an integer $h>0$, and $\mathbb{F}_q$ a finite field with $q$ elements. Denote its multiplicative group by $\mathbb{F}_q^*$. For a positive integer $n\in\Z_{>0}$, we write $[n]\coloneqq\{1,...,n\}$.

\subsection{Codes}

First, we recall some notions on coding theory.
\begin{definition}
    If $x=(x_1,x_2,\dots, x_n)$ and $y=(y_1,y_2,\dots, y_n)$ are vectors in $\F_q^n$, their \textit{Hamming distance} is the integer $d(x,y):=|\{i:x_i\neq y_i\}|$
    given by the number of their distinct entries. The \textit{Hamming weight} $w(x)$ of a vector $x\in\F_q^n$ is the number of its nonzero entries, that is, $w(x):=d(x,0)$,
    and the \textit{minimum distance} of a subset $S\subseteq \F_q^n$ is $d(S):=\min\{d(x,y): x\neq y \text{ in }S\}.$
\end{definition}

\begin{definition}
    Let $n,k,d\in \mathbb{Z}_{>0}$ such that $k,d\leq n$. A \textit{linear code} $C$ over $\F_q$, of length $n$, dimension $k$ and minimum distance $d$, is an $\mathbb{F}_q$-linear subspace of $\F_q^n$ with $\dim C =k$ and $d(C)=d$. We refer to $C$ as an $[n,k,d]$-code and to its elements as \textit{codewords}.
\end{definition}
 Note that, as a vector space, $C$ has minimum distance equal to its minimum Hamming weight, i.e., $d=\min\{d(x,0):0\neq x\in C\}.$
Moreover, we can obtain a new code from $C$ by restricting to a subset of code positions.

\begin{definition}
    If $I=\{a_1,a_2,\dots, a_{n_1}\}\subseteq[n]$, with $n_1\leq n$ and $a_i< a_{i+1}$ for all $i\in [n_1]$,  define the puncture of a codeword $x$ on $I$ as
    $x_I=(x_{a_1},x_{a_2},\dots, x_{a_{n_1}})$  and the puncture of $C$ on $I$ as $C_I=\{x_I:x\in C\}$. We call $I$ the support of $C_I$, relative to $C$, and denote it by $\supp C_I$.
\end{definition}
In the literature, the term puncturing sometimes assumes that the resulting code has the same dimension as the original. We do not make this assumption here. It is not hard to see that the puncture $C_I$ of an $[n,k,d]$-code $C$ is again a linear code and that $C_I$ is an $[n_1,k_1,d_1]$-code with $n_1\leq n$, $k_1\leq k$ and $d_1\leq d$.
The notion of punctured codes is central to locally recoverable codes. 

\begin{definition} [\cite{PKLK12}]\label{LRCdef2} Let $n, r, \rho\in \mathbb{Z}_{>0}$. A linear code $C$ of length $n$ over $\mathbb{F}_q$ is \textit{$(r,\rho)$-locally recoverable} if, for each $i\in[n]$, there exists a punctured code $C_i$ of length at most $r+\rho-1$ and minimum distance at least $\rho$, such that the support of $C_i$ contains $i$. The support of $C_i$ and $\supp C_i \setminus \{i\}$ are respectively called a \textit{repair group} and a \textit{recovery set} for position $i$. Sometimes we say that $C$ has $(r,\rho)$-\textit{locality}.
\end{definition}

Note that $\supp C_i$ is a repair group for all $j\in\supp C_i$ and that the dimension of the punctured code is at most $r$, because $r$ coordinates are sufficient to determine the other ones.
This definition of a locally recoverable code generalizes the idea in the introduction by allowing the recovery of $\rho-1$ erasures within any repair group.  The single erasure case is covered by $\rho=2$.  
Note here that $r$ is the dimension of the lower code $C_i$ and $\rho$ is its minimum distance.

Locally recoverable codes with hierarchy were first studied in \cite{sasidharan2015codes}. We present here the definition of HLRCs introduced there, with some slight modifications.

\begin{definition} [\cite{sasidharan2015codes}]\label{HLRCdef}
Let $n, n_1, n_2, k, k_1, k_2, d, d_1, d_2\in\mathbb{Z}_{>0}$ with $n_2<n_1\leq n$, $k_2\leq k_1\leq k$ and $d_2<d_1\leq d$.  A linear $[n,k,d]$-code $C$ is said to have \emph{hierarchical locality} with local parameters $((n_1, k_1, d_1), (n_2,  k_2, d_2))$ if, for each $i\in[n]$, there exists a punctured code $C_{1,i}$ of length at most $n_1$, $\dim C_{1,i}\leq k_1$ and minimum distance $d(C_{1,i})\geq d_1$, with $i \in \supp C_{1,i}$, such that $C_{1,i}$ has $ (k_2,d_2)$-locality with repair groups of size at most $n_2$. We denote by $C_{2,i}$ the puncture of $C_{1,i}$ on its repair group for position $i$, so that $C_{2,i}$ has length at most $n_2$, $\dim C_{2,i} \leq k_2$, $d(C_{2,i})\geq d_2$ and $i\in \supp C_{2,i}$.

Locally Recoverable Codes with hierarchical locality are referred to as \emph{HLRCs}. 
\end{definition}

\begin{remark}
    We have listed $n_1$ and $n_2$ to be upper bounds on the lengths the of the corresponding codes to encompass the situation in which different positions have recovery groups of different sizes.  The maximum smaller recovery group size for any position would be $n_2$ and the maximum length of a punctured code $C_{1,i}$ for any $i$ would be $n_1$.
\end{remark} 
\begin{remark}
    \label{rem:middle:lower:codes}
    Throughout the paper, we refer to the code $C_{1,i}$ (resp. $C_{2,i}$) in \cref{HLRCdef} as the \emph{middle} (resp. \emph{lower}) \emph{code} for position $i$. We choose this terminology, which differs from \cite{sasidharan2015codes}, for a more convenient discussion in our exposition, as it will be clear in the remaining sections of the paper.
\end{remark}
 
We conclude this subsection by recalling the definition of evaluation codes on a variety $\calS$ over $\mathbb{F}_q$, as this is the kind of codes we deal with in the paper.
\begin{definition}
    \label{defn:evalcode}
    Let $\calS$ be a quasi-projective variety over $\mathbb{F}_q$. 
    Consider a set $ T=\{P_1, \dots, P_n\}$
    of points in $\calS(\mathbb{F}_q)$, and an $\mathbb{F}_q$-vector space of functions in the function field $\mathbb{F}_q(\calS)$ with no poles in $T$. The \emph{evaluation code} of $V$ in $T$ is the $\mathbb{F}_q$-vector space
    \[
    C(T,V)\coloneqq \{(f(P_1),\dots, f(P_n)):f\in V\}\subseteq \mathbb{F}_q^n.
    \]
In other words, $C(T,V)$ is defined as the image of the evaluation map
    $ev_T : V \to \mathbb{F}_q^n$, where for each $f\in V$, $ev_T(f) := (f(P_1),\dots, f(P_n))$.
\end{definition}
Note that the codes defined in \cref{defn:evalcode} are in the class of Algebraic Geometry (AG) codes, since they are given by evaluating functions in the function field of an algebraic variety at points of the variety.

\subsection{Algebraic geometric tools}
\label{subsec:agtools}
In this paper, the evaluation sets defining our codes consist of points on surfaces in the affine space $\A^3_{(x,y,z)}$, with coordinates $(x,y,z)$.  However, an important purpose of this work is to use projective spaces and evaluate functions on points at infinity to improve code parameters. For $\gamma\in\F_q$, we identify the plane $z=\gamma$ in $\A^3_{(x,y,z)}$ with the 2-dimensional affine space $\mathbb{A}^2_{(x,y)}$ and refer to it as $\A^2_{\gamma}$.  Then, we take the projective closure of $\A^2_{\gamma}$, denote it by $\mathbb{P}^2_\gamma$ and use homogeneous coordinates $(x:y:w)$.
Another fundamental tool is to consider fibers of projection maps. To denote these maps, we use $\pi$ with subscripts indicating the coordinates to which they project. For example, define
\begin{align*}
 \pi_x : \mathbb{A}^3_{(x,y,z)} &\to \mathbb{A}^1_x & \pi_{x,z}: \mathbb{A}^3_{(x,y,z)} &\to \mathbb{A}^2_{(x,z)} \\
     (\alpha,\beta,\gamma)&\mapsto \alpha     &     (\alpha,\beta,\gamma)&\mapsto (\alpha,\gamma)
\end{align*}

When working on $\mathbb{P}^2_{\gamma}$, some basic intersection theory will be useful. The following definition and theorem are both from \cite[Section 5.14]{GW}.
\begin{definition}
    Let $K$ be any field. Consider two plane curves $C,D\subset \mathbb{P}^2(K)$ such that their intersection $X= C\cap D$ has dimension $0$. Let $\Gamma(X,\mathcal{O}_X)$ denote the space of global sections of the structure sheaf $\mathcal{O}_X$ of $X$. 
    Then we call $i(C,D)\coloneqq \dim_K\Gamma(X,\mathcal{O}_X)$ the \emph{intersection number} of $C$ and $D$.
    For $P\in X$, we call $i_P(C,D)\coloneqq \dim_K\mathcal{O}_{X,P}$ the \textit{intersection number} of $C$ and $D$ at $P$, where $\mathcal{O}_{X,P}$ denotes the local ring of $X$ at $P$.
\end{definition}

\begin{remark}
    Observe that $\dim X =0$ if and only if $C$ and $D$ are given by homogeneous polynomials without a common factor. In this case, we have
    \[
    i(C,D)=\sum_{P\in X}i_P(C,D).
    \]
\end{remark}

\begin{theorem}[B\'ezout]
    If $C$ and $D$ are plane curves in $\mathbb{P}^2(K)$ given by homogeneous polynomials $g$ and $h$ without a common factor, then $i(C,D)=(\deg g)(\deg h).$
\end{theorem}

\subsection{Artin-Schreier and Kummer Covers}\label{ss:ArtinSchreier&KummerCovers}

In this subsection, we briefly recall some concepts in the language of one-variable function fields, since this paper makes extensive use of two special types of Galois function field extensions. We refer to \cite{Sti} for a comprehensive exposition of the theory of one-variable function fields.

Let $K$ be any perfect field and $x$ be a transcendental element over $K$. An \emph{algebraic function field} $F/K$ \emph{of one variable} over $K$, often simply called a \emph{function field}, is a finite extension of $K(x)$. The field $K$ is called the \emph{constant field} of $F$, and when it is understood, we usually just write $F$ instead of $F/K$. Let now $\Phi$ be a fixed algebraically closed field containing $F$, and let $F'/K'$ be some other function field such that $\Phi \supseteq F'$.

The function field $F'/K'$ is said to be an \emph{algebraic extension} of $F/K$ if $F' \supseteq F$ is an algebraic field extension and $K' \supseteq K$. As in \cite{Sti}, we adopt the notation $F'/F$ for function field extensions.

A fundamental invariant of a function field $F/K$ is its \emph{genus} $g$, which coincides with the dimension, as a vector space over $K$, of the space of regular differentials of $F$. Note that this is exactly the dimension of the Riemann-Roch space associated to any canonical divisor of $F$. 

Function fields having many rational places with respect to their genus, have been, and still are, extensively studied in the literature, not only because they are particularly interesting as extremal objects, but also because of their applications in the construction of evaluation codes (AG codes) with good parameters. The celebrated Hasse-Weil bound gives both a lower and an upper bound on the number of rational places of a function field over $\Fq$.
\begin{theorem}[Hasse-Weil Bound]\label{thm:HasseWeil}
    Let $F/\mathbb{F}_q$ be a function field of genus $g$ over the finite field $\mathbb{F}_q$. The number $N(F)$ of $\mathbb{F}_q$-rational places of $F$ satisfies the inequality
    \[
    |N(F)-(q+1)|\leq 2g\sqrt{q}.
    \]
\end{theorem}
A function field $F/\mathbb{F}_q$ of genus $g$ is said to be \emph{maximal} if $N(F)=q+1+2g\sqrt{q}$. The (nonsingular) algebraic curve associated to a maximal function field is called \emph{maximal} as well.

\begin{example}
\label{ex:hermitian}
    One of the most famous examples of maximal function fields is the \emph{Hermitian function field}, that is, the $\F_{q^2}$-rational function field of the Hermitian curve. 

    Consider the affine model of the Hermitian curve $\mathcal{H}\subseteq \mathbb{A}^2_{(x,y)}$ given by the affine equation $y^{q+1}=x^q + x$
    over $\mathbb{F}_{q^2}$. The $\F_{q^2}$-rational function field of $\mathcal{H}$ is $H\coloneqq\mathbb{F}_{q^2}(x,y)$, with $y^{q+1}=x^q + x$, and it is a maximal function field of genus $g=q(q-1)/2$ over $\mathbb{F}_{q^2}$.
\end{example}

We now recall the definitions of two special types of Galois function field extensions, and we introduce the terminology of \emph{Kummer-type} and \emph{Artin-Schreier} surfaces, which we use throughout the paper.

\begin{definition}
    Let $\lambda$ be a natural number. Let $F/K$ be an algebraic function field where $K$ contains a primitive $\lambda$-th root of unit (with $n>1$ and $n$ relatively prime to the characteristic of $K$).  Suppose that $u\in F$ is an element satisfying $u\neq w^d$ for all $w\in F$ and $d|\lambda$, $d>1$.  Let $F'=F(y)$ with $y^{\lambda}=u$.  Then $F'/F$ is called a \textit{Kummer extension} of $F$.
\end{definition}

\begin{theorem}[{\cite[Proposition 3.7.3]{Sti}}]
With notation as above, let $F'/F$ be a Kummer extension of $F$.
\begin{itemize}
\item The extension $F'/F$ is Galois of degree $n$, with minimal polynomial $T^{\lambda}-u$.  The Galois group is cyclic, and the automorphisms of $F'/F$ are given by $\sigma(y)=\zeta y$, where $\zeta\in K$ is a $\lambda$-th root of unity.  
\item Let $P$ be a place of $F$ with $P'$ a place of $F'$ lying above $P$.  Then the ramification index of $P'$ is $e(P'|P)=\frac{\lambda}{r_P}$ and the different exponent is $d(P|P')=\frac{\lambda}{r_P}-1$, where $r_P:=\gcd(\lambda,v_P(u))$.
\end{itemize}
\end{theorem}

If the characteristic of $F$ is relatively prime to $\lambda$ and $F$ contains a primitive $\lambda$-th root of unity, then all cyclic Galois degree-$\lambda$ extensions of $F$ are Kummer extensions.  

In the case that $F=K(x)$ is a rational function field, and $u=f(x)\in F$ satisfies the conditions above, then the curve $C$ with affine model $y^\lambda=f(x)$ is known as a \textit{Kummer curve}, and the induced covering map $C\rightarrow\P^1_x$ is called a \textit{Kummer cover}.

Note that there are exactly $\lambda$ distinct points of $C$ lying above any unramified point of $\P^1_x$, and, considering that $\zeta\in K$, if any of these points are defined over a given extension of $K$ then all are.

If $F=K(x,z)$ is the function field of $\P^2_{x,z}$, and $u=f(x,z)$ satisfies the conditions above, then the surface $S$ with function field $F'$ and affine model $y^{\lambda}=f(x,z)$ is called a \textit{Kummer-type surface}.

\begin{definition}
    Let $F/K$ be an algebraic function field of characteristic $p>0$.  Suppose that $u\in F$ is an element that satisfies $u\neq w^p-w$ for all $w\in F$ and $F'=F(y)$ where $y^p-y=u$.  Then $F'/F$ is called an \textit{Artin-Schreier extension} of $F$.
\end{definition}
    
\begin{theorem}[{\cite[Proposition 3.7.8]{Sti}}]
With notation as above, let $F'/F$ be an Artin-Schreier extension. For each $P$ a place of $F$, define $m_P:=m$ if there exists some $z\in F$ satisfying $v_P(u-(z^p-z))=-m<0$ with $m\not\equiv0\pmod{p}$, and $m_P:=-1$ otherwise, that is, if $v_P(u-(z^p-z))\geq 0$ for some $z\in F$.  
\begin{itemize}
    \item The extension $F'/F$ is Galois of degree $p$. The automorphisms of $F'/F$ are given by $\sigma(y)=y+\nu$ for $\nu=0,1,\dots, p-1$.
    \item A place $P$ of $F$ is unramified in the extension if and only if $v_P(u-(z^p-z))\geq 0$ for some $z\in F$.
    \item The place $P$ is totally ramified if and only if there exists $z\in F$ satisfying $v_P(u-(z^p-z))<0$
\end{itemize}

\end{theorem}

If $F$ has characteristic $p$, any cyclic field extension of degree $p$ is an Artin-Schreier extension.  

In the case that $F=K(x)$ is a rational function field, and $u=f(x)\in F$ satisfies the conditions above, then the curve $C$ with affine model $y^p-y=f(x)$ is known as an \textit{Artin-Schreier curve}, and the induced covering map $C\rightarrow\P^1_x$ is called an \textit{Artin-Schreier cover}.

Note that, considering the Artin-Schreier automorphisms, there are exactly $p$ points of $C$ lying above any unramified point of $\P^1_x$ in the cover, and if any of these $p$ points are defined over a given extension of $\F_p$, then all are.

If $F=K(x,z)$ is the function field of $\P^2_{x,z}$, and $u=f(x,z)\in F$ satisfies the conditions above, then the surface $S$ with function field $F'=F(y)$ and affine model $y^p-y=f(x,z)$ is called an \textit{Artin-Schreier surface}.

Our objects of study are surfaces with fibrations of Artin-Schreier and Kummer curves. The study of codes constructed from Artin–Schreier and Kummer curves is extensive, however, their higher-dimensional analogues remain comparatively underdeveloped. In \cite{berg2024codes}, the authors introduced a novel use of geometry and fibrations of surfaces into curves to obtain locality and hierarchy for codes. Their framework allows us to adapt the construction by studying the surface in a quasi-projective framework, which leads to HLRCs with improved minimum distance.

\section{Codes on Artin-Schreier surfaces}
\label{sec:AS}

In this section, we present a construction of HLRCs from Artin-Schreier surfaces.  With our approach, which revisits the one in \cite{berg2024codes}, we give a non-trivial lower bound for the minimum distance of the obtained codes that, under certain assumptions, improves the bound contained in \cite{berg2024codes}. When our bound is not better, we are still able to retrieve the bound from \cite{berg2024codes}. 

\subsection{Code construction: a quasi-projective approach}
\label{sub:ASgeneral} 

Let $q\coloneqq p^h$ be a positive power of some prime $p$.
Consider the surface in $\mathbb{A}^3_{(x,y,z)}$ defined by
\begin{equation}
\label{eq:AS:surface}
    \calS_f : y^p -y-f(x,z) =0,
\end{equation}
where $f\in \mathbb{F}_q[x,z]$ is a polynomial with $\deg_x f\geq p+1$ and $f(x,z)\neq g(x,z)^p-g(x,z)$ for any $g(x,z)\in \overline{\F_q}(x,z)$.
Furthermore, for $\gamma\in\mathbb{F}_q$, define the variety
\begin{equation}
\label{eq:AS:Zgamma}
    \Zgamma\coloneqq \begin{cases}
        y^p-y-f(x,z)=0\\
        z-\gamma=0
    \end{cases} \ ,
\end{equation}
which is the intersection of $\calS_f$ with the plane $\mathbb{A}^2_\gamma$, given by $z=\gamma$, in $\mathbb{A}^3_{(x,y,z)}$, see \cref{subsec:agtools}. In other words, generically, $\Zgamma$ is the curve defined by $y^p-y-f(x,\gamma)=0$ in $\mathbb{A}^2_\gamma$. Note that in some cases, the intersection will not be irreducible, but the fibers that we make use of will be geometrically irreducible curves.
As described in \cref{subsec:agtools}, we denote the projective closure of $\mathbb{A}^2_\gamma$ by $\mathbb{P}^2_\gamma$, with homogeneous coordinates $(x:y:w)$. Furthermore, we let $s\coloneqq \deg_xf$ and we denote by $F_\gamma(x,w)$ the homogenization of $f(x,\gamma)\in \mathbb{F}_q[x]$ with respect to the variable $w$. Then, the curve defined by 
\begin{equation}\label{eq: projective AS}
    \overline{\calZ}_\gamma:y^p w^{s-p}-yw^{s-1}-F_\gamma(x,w) = 0
\end{equation}
is the projective closure of $\Zgamma$ in $\mathbb{P}^2_\gamma$. Observe that $\ZZ$ has a unique point at infinity $P_\gamma\coloneqq(0:1:0)$ in $\mathbb{P}^2_\gamma$, which corresponds to the point $[(0:1:0),(\gamma, 1)]$ in the partial compactification of $S_f$ in $\mathbb{P}^2_{(x:y:w)}\times \mathbb{A}^1_z$.

In the following lemma, we compute the multiplicity of $P_\gamma$ on $\overline{\calZ}_\gamma$, together with the intersection multiplicities of $\overline{\calZ}_\gamma$ with the lines $w=0$ and $x=0$ at $P_\gamma$.

\begin{lemma}
\label{lemma:singularity}
If $s\geq p+1$, then $P_{\gamma}$ is a point in $\overline{\calZ}_\gamma$ of multiplicity $s-p$. Moreover, the line $w=0$ is the only tangent of $\overline{\calZ}_\gamma$ at this point and intersects it with multiplicity $s$, while the line $x=0$ intersects $\overline{\calZ}_\gamma$ at $P_\gamma$ with multiplicity $s-p$.
\end{lemma}

\begin{proof}
Let $h(x,y,w) = y^pw^{s-p}-yw^{s-1} - F_\gamma(x,w)$ be the homogeneous polynomial defining $\overline{\calZ}_\gamma$ in $\mathbb{P}^2_\gamma$. Observe that $F_\gamma(x,w)$ is a sum of monomials (in $x$ and $w$) each of degree $s$. Since the point $P_\gamma$ is the origin in the affine chart $y=1$, to compute its multiplicity on $\overline{\calZ}_\gamma$ we dehomogenize $h(x,y,w)$ with respect to $y$, obtaining $w^{s-p}- w^{s-1} - F_\gamma(x,w)$, and look at the term of lowest degree in this expression. Since all the monomials appearing in $F_\gamma(x,w)$ are of degree $s$, it follows that the lowest-degree term is $w^{s-p}$, which shows that $P_\gamma$ is a point of multiplicity $s-p$ in $\overline{\calZ}_\gamma$ and that the only tangent of $\overline{\calZ}_\gamma$ at $P_\gamma$ is the line $w=0$. In particular, notice that if $s=p+1$ then $P_\gamma$ is a nonsingular point of $\overline{\calZ}_\gamma$ with tangent line $w=0$, while if $s>p+1$ then $P_\gamma$ is singular, but still only has $w=0$ as a tangent. 

From the equation defining $\overline{\calZ}_\gamma$, we see that $P_\gamma$ is the only point of the intersection of $\overline{\calZ}_\gamma$ with the line $w=0$. By B\'ezout theorem, the intersection multiplicity of $w=0$ and $\overline{\calZ}_\gamma$ at $P_\gamma$ is precisey $s$, the degree of $\overline{\calZ}_\gamma$. On the other hand, for any other line $\ell$ through $P_\gamma$, since $\ell$ is not tangent to $\overline{\calZ}_\gamma$ at $P_\gamma$, the intersection multiplicity of $\ell$ and $\overline{\calZ}_\gamma$ at $P_\gamma$ is $s-p$, the multiplicity of $P_\gamma$ as a point of $\overline{\calZ}_\gamma$.
\end{proof}

We now define a family of HLRCs as evaluation codes from the surface $\mathcal{S}_f$.
\begin{definition} \label{def:main:artinschreier} 
Let $f\in\F_{q}[x,z]$ with $\deg_x(f(x,z))=s$, where $s$ is an integer $s\geq p+1$.
    Let $\calS_f$ be the surface defined by $y^p-y=f(x,z)$ as in \cref{eq:AS:surface}, and $\Zgamma$ be defined as in \cref{eq:AS:Zgamma}. Let $\eta$ be a positive integer with \[0\leq \eta \leq \max\{|\pi_x(\Zgamma(\F_q))|:\gamma\in\F_q,\deg_x(f(x,\gamma))=s\}\] and let $\Gamma\subseteq\F_q$ denote the subset 
    \begin{equation}\label{eqn:Gamma}
        \Gamma := \{ \gamma \in \F_q : 
        \pi_x(\Zgamma(\F_q))\geq \eta \textrm{ and }\deg_x(f(x,\gamma))=s
        \}.
    \end{equation}

    Let $\rho_1\in[\eta]$, $\rho_2\in [p]$, and $\rho_3\in[|\Gamma|]$ be positive integers such that
    \begin{enumerate}
    \item $\rho_1 \ge 2$, $\rho_2 \ge 2$, and $\rho_3\ge 1$;
    \item $p\eta\geq s(\eta - \rho_1 + p-\rho_2)-(s-p)(\eta-\rho_1)+1=s(p-\rho_2) +p(\eta-\rho_1)+1$.
    \end{enumerate}

    Define an evaluation code (see \cref{defn:evalcode}) $C^{f} := C(T,V)$ with evaluation set
    \begin{equation}\label{eqn:T}
        T:=\{(x,y,z)\in \calS(\F_q): z\in\Gamma\} = \bigcup_{\gamma\in\Gamma}\Zgamma(\mathbb{F}_q),
    \end{equation}
and vector space of evaluation polynomials
    \begin{equation}\label{eqn:V}
        V=V_{\rho_1,\rho_2,\rho_3}:=\langle x^iy^jz^k:
        0\leq i\leq \eta-\rho_1, 
        0\leq j\leq p-\rho_2, 
        0\leq k \leq |\Gamma|-\rho_3 
        \rangle.
    \end{equation} 
    \end{definition}
    
    In the following subsection, we prove that a code $C^{f}$ has locality and hierarchy, and we discuss its dimension, minimum distance and rate.

\begin{remark}
    \label{rem:AS:length}
    By construction, the length of a code $C^f$ as in \cref{def:main:artinschreier} is $n\coloneq |T|\geq p \eta |\Gamma|$. Indeed, let $\eta$ be such that each $\Zgamma$ has at least $\eta p$ affine points. Since $\Zgamma$ is an Artin-Schreier curve, this is equivalent to having $\eta$ or more distinct $x$-coordinate values for the affine points of $\Zgamma$. Therefore, $|T|\geq p \eta |\Gamma|$, since $|\pi_x(\Zgamma)|\geq \eta$ and $\pi_z(T)=\Gamma$.
\end{remark}

\begin{remark}
    \label{rem:AS:rho1}
    As it will be clear from the discussion in \cref{subsec:AS:parameters}, condition $(2)$ in \cref{def:main:artinschreier} is needed to ensure that a certain lower bound on the minimum distance $d$ of $C^f$ is positive. However, it should be noted that, even if the condition is dropped, a non-trivial (positive) lower bound on the minimum distance still holds. More precisely, in the proof of \cref{thm:AS:parameters}, we show that the lower bound $d \geq \rho_1\rho_2\rho_3$ always holds, and from \cref{rem:mindistbounds:comparison} we deduce that, if $\rho_1 \geq s$, then the bound proved in \cref{cor:mindist} is larger than or equal to $\rho_1\rho_2\rho_3$. 
\end{remark}

\subsection{Dimension, locality and lower bounds for the minimum distance} 
\label{subsec:AS:parameters}

In this subsection, we discuss the parameters of a code $C^{f}$ as in \cref{def:main:artinschreier}. We start by computing a bound on the minimum distance $d$ of $C^{f}$ by applying B\'ezout's theorem on each fiber $\ZZ$ and subtracting the minimum vanishing multiplicity of the functions in $V$ at $P_\gamma$, for all $\gamma\in \Gamma$. 

Let $M\coloneqq\eta + p -\rho_1-\rho_2$ and consider the $\mathbb{F}_q$-vector space of (homogeneous) forms
\begin{equation}
\label{eq:V:homog}
    V_H=\langle x^iy^jw^\ell:0\leq i\leq \eta-\rho_1, 0\leq j\leq p-\rho_2, i+j+\ell=M\rangle.
\end{equation}
Note that, when $g\in V$ is restricted to $\mathbb{A}^2_\gamma$, it yields a function $g_{\gamma}(x,y):=g(x,y,\gamma)$, which can be written as $g_{\gamma}=G_\gamma(x,y,1),$
where $G_\gamma\in V_H$ is the homogenization of $g_\gamma$ with respect to some variable $w$. It is then clear that the vanishing of $G_\gamma$ on affine points of $\ZZ$ with $w=1$ is the same as that of $g_{\gamma}$ on $\Zgamma$. Furthermore, the vanishing multiplicity of $g\in V$ at $P_\gamma$ is also described by the vanishing multiplicity of $G_\gamma$ at $P_\gamma$, which is exactly the intersection multiplicity of the curve defined by $G_\gamma$ and $\ZZ$ at the origin, in the affine chart with $y=1$ of $\P^2_\gamma$. In the light of this, the following lemma computes the minimum vanishing multiplicity of any function in $V$ at $P_\gamma$, for any $\gamma\in \Gamma$.

\begin{lemma}
\label{lemma:vanishing:at:infty}
    Let $\ZZ$ and $P_\gamma$ be as defined above, and let $V_{H}$ be the $\mathbb{F}_{q}$-vector space of (homogeneous) forms defined in \cref{eq:V:homog}.
    Let $\mathcal{G}$ denote the curve in $\mathbb{P}^2_\gamma$ defined by the vanishing of the form $G \in V_{H}\setminus \{0\}$.
    Then 
    \begin{equation*}
        \mathrm{min}_{G\in V_{H}\setminus \{0\}}\{i_{P_\gamma}(\mathcal{G}, \ZZ)\}=(s-p)(\eta-\rho_1).
    \end{equation*}
\end{lemma}

\begin{proof}
    Let
    \begin{equation*}
        G(x,y,w)=\sum_{i=0}^{\eta-\rho_1}\sum_{j=0}^{p-\rho_2}a_{i,j}x^iy^jw^{M-i-j},
    \end{equation*}
    where $a_{i,j}\in \F_q$, and at least one $a_{i,j}\neq 0$.

    We start by observing that, by \cref{lemma:singularity}, the degree-1 form $x$ (which defines the line $x=0$ in $\P^2_\gamma$) vanishes with multiplicity $s-p$ at the point $P_\gamma$ of $\ZZ$, and the degree-1 form $w$ vanishes with multiplicity $s$ at the point $P_\gamma$ of $\ZZ$.  Thus if $a_{i,j}\neq 0$, the monomial $a_{i,j}x^iy^jw^{M-i-j}$ vanishes with multiplicity $i(s-p)+(M-i-j)s$, and the degree-$M$ form $G$ vanishes with multiplicity at least
\[\min\{i(s-p)+(M-i-j)s:0\leq i\leq\eta-\rho_1, 0\leq j\leq p-\rho_2\}\] at $P_{\gamma}$. Since $s-p<s$, this minimum will be $(s-p)(\eta-\rho_1)$.  
\end{proof}

\begin{corollary}
\label{cor:vanishing}
Every non-zero function in $V_{H}$ vanishes in at most $Ms-(s-p)(\eta-\rho_1)$ points of $\ZZ$ aside from $P_{\gamma}$.    
\end{corollary}

\begin{proof}
    This is a direct application of Bez\'out's theorem.
\end{proof}

\begin{remark} \label{rem:mindistbounds:comparison}
    Let $c\in C^f$ be any nonzero codeword, and consider the punctured code obtained from $C^f$ by retaining the codeword positions corresponding to the points in the set $\{P\in T \mid \pi_z(P)=\gamma\}$, which are precisely the $\F_q$-points on $\mathcal{Z}_{\gamma}$. Observe that from \cref{cor:vanishing} it follows, in particular, that the minimum distance of such punctured code is at least $\eta p-Ms+(s-p)(\eta-\rho_1)$. However, in the proof of \cref{thm:AS:parameters}, we show that another lower bound is $\rho_1 \rho_2$. Depending on $\rho_1$ and $s$, we can determine which of these two lower bounds is larger. Indeed, note that $\eta p-Ms+(s-p)(\eta-\rho_1) = \rho_1 p - s (p-\rho_2)$, and $\rho_1 p - s (p-\rho_2)\geq \rho_1\rho_2$ if and only if $(p-\rho_2)(\rho_1-s)\geq 0$, hence if and only if $\rho_1 \geq s$ since $p-\rho_2\geq 0$ always holds by the assumptions in \cref{def:main:artinschreier}. 
\end{remark}

\begin{corollary}
\label{cor:mindist}
    The minimum distance of $C^{f}$ is at least $\rho_3\left(\eta p-Ms+(s-p)(\eta-\rho_1)\right)= \rho_3(\rho_1 p - s(p-\rho_2))$.
\end{corollary}
\begin{proof}
Let $g(x,y,z)\in V\setminus\{0\}$.  As there are $|\Gamma|$ distinct $z$-values for points of the evaluation set $T$, and the maximum degree of $g$ in $z$ is $|\Gamma|-\rho_3$, there must be at least $\rho_3$ values of $\gamma\in\Gamma$ so that $g(x,y,\gamma)$ is not identically 0.  For each of these values of $\gamma$, $g$ has the same vanishing on $\Zgamma$ as the corresponding $G_\gamma\in V_{H}$ on affine points of $\ZZ$.  Since \cref{cor:vanishing} implies that $G_{\gamma}$ vanishes on at most $Ms-(s-p)(\eta-\rho_1)$ affine points of $\ZZ$, in particular $G_{\gamma}$ vanishes on at most this number of points of $\Zgamma$.  Therefore there must be $\eta p-Ms+(s-p)(\eta-\rho_1)$ evaluation points of $\Zgamma$ where $g$ does not vanish. By condition (2) in \cref{def:main:artinschreier}, this quantity will be at least 1.
\end{proof}

\begin{remark}
    Observe that in the case where \(j = p-\rho_2\) and \(\eta = \rho_1\), the function \(y^{p-\rho_2}\) does not vanish at \(P_\gamma\), and the intersection multiplicity given in \cref{lemma:vanishing:at:infty} is zero. In this situation, \cref{cor:vanishing} simply recovers the conclusion of Bézout’s Theorem and the result still holds. The bound on the minimum distance in this case is $\rho_3(\eta p - Ms)$. By \cref{def:main:artinschreier}, $\eta p - Ms \geq 1$ and $\rho_3\geq 1$, hence the minimum distance is positive also in this case.

\end{remark}

\cref{cor:mindist} yields the lower bound $d \geq \rho_3\left(\eta p-Ms+(s-p)(\eta-\rho_1)\right)$ for a code $C^f$ as in \cref{def:main:artinschreier}. We now show that any such code is an HLRC with lower codes supported on the fibers of $\pi_{x,z} \colon \calS_f \to \mathbb{A}^2_{x,z}$ and middle codes supported on the curves $\Zgamma$, see \cref{HLRCdef} and \cref{rem:middle:lower:codes}.

\begin{theorem}
\label{thm:AS:parameters}
    Let $f\in\F_{q}[x,z]$ with $\deg_x(f(x,z))=s$, where $s\in \mathbb{Z}_{\geq p+1}$,
    and let $\calS_f$ and $\Zgamma$ be as in \cref{eq:AS:surface} and \cref{eq:AS:Zgamma}. Let $C^f$ be an evaluation code from $\mathcal{S}_f$ as in \cref{def:main:artinschreier}, obtained with any suitable choice of the parameters $\eta, \rho_1, \rho_2, \rho_3$. 

    The code $C^f$ has hierarchical locality with lower codes supported on the fibers of $\pi_{x,z} \colon \calS_f \to \mathbb{A}^2_{(x,z)}$ and middle codes supported on the curves $\Zgamma$. More precisely, $C^f$ is an $[n,k,d]$-code with
    \begin{align*}
        n&=|T|\geq \eta p|\Gamma|, \quad
        k=(\eta-\rho_1+1)(p-\rho_2+1)(|\Gamma|-\rho_3+1),  \\
        d&\geq \rho_3\max\{\rho_1\rho_2, \ \rho_1 p - s(p-\rho_2)\}.
    \end{align*}
    Moreover, with notations as in \cref{HLRCdef}, $C^f$ has local parameters 
    \begin{equation*}
        ((n_1,k_1,d_1),(n_2,k_2,d_2)),
    \end{equation*}
    where
    \begin{align*}
        n_1&\geq \eta p, &
        k_1&=(\eta-\rho_1+1)(p-\rho_2+1), &
        d_1& = \max\{\rho_1\rho_2, \ \rho_1 p - s(p-\rho_2)\} \\
        n_2&=p, &
        k_2&=p-\rho_2 +1, &
        d_2&= \rho_2.
    \end{align*}
\end{theorem}

\begin{proof}
                                                                                                                                                                                                                                                                                                                                                                                                                                                                                                                                                                                                                                                                                                                                                                                                                                                                                                                                                                                                                                                                                                                                                                                                                                                                                                                                                                                                                                                                                                                                                                                                                                                                                                                                                                                                                                                                                                                                                                                                                                                                                                                                                                                                                                                                                                                                                                                                                                                                                                                                                                                                                                                                                                                                                                                                                                                                                                                                                                                                                                                                                                                                                                                                                                                                                                                                                                                                                                                                                                                                                                                                                                                                                                                                                                                                                                                                                                                                                                                                                                                                                                                                                                                                                                                                                                                                                                                                                                                                                                                                                                                                                                                                                                                                                                                                                                                                                                                                                                                                                                                                                                                                                                                                                                                                                                                                                                                                                                                                                                                                                                                                                                                                                                                                                                                                                                                                                                                                                                                                                                                                                                                                                                                                                                                                                                                                                                                                                                                                                                                                                                                                                                                                                                                                                                                                                                                                                                                                                                                                                                                                                                                                                                                                                                                                                                                                                                                                                                                                                                                                                                                                                                                                                                                                                                                                                                                                                                                                                                                                                                                                                                                                                                                                                                                                                                                                                                                                                                                                                                                                                                                                                                                                                                                                                                                                                                                                                                                                                                                                                                                                                                                                                                                                                                                                                                                                                                                                                                                                                                                                                                                                                                                                                                    
    Consider the position with index $i$ of a codeword $c\in C^f$. The coordinate of $c$ in position $i$ is given by the evaluation of a polynomial in $V$ at some fixed point $P_i\coloneqq (\alpha, \beta, \gamma)\in T$. Denote by $C^f_{2,i}$ the punctured code obtained from $C^f$ by retaining the codeword positions corresponding to the points in the evaluation set $T$ having the same $x$- and $z$-coordinates as $P_i$, i.e. with $x=\alpha$ and $z=\gamma$. These points are exactly those in the fiber
        \begin{equation*}
            F_{2,i} \coloneqq  (\pi_{x,z}\mid_{\calS_f})^{-1}(\pi_{x,z}(P_i))=\{(\pi_x(P_i),\pi_y(P_i)+\delta,\pi_z(P_i)):\delta\in\F_p\},
        \end{equation*}
        which in particularly shows $n_2=p$.
    Let now $g\in V$ and observe that $g|_{F_{2,i}}=g(\alpha, y, \gamma)$ is a univariate polynomial in $y$ of degree at most $p-\rho_2$. 
    This means that, knowing the evaluation of $g$ at $p-\rho_2+1$ points in $F_{2,i}$, we can interpolate $g(\alpha,y,\gamma)$ across all of $F_{2,i}$ using Lagrange interpolation. 
    Therefore, we have local recovery with $(p-(\rho_2-1),\rho_2)$-locality and any $\rho_2-1$ erasures of coordinates corresponding to points in $F_{2,i}$ can be recovered at the lower level.  This implies that the lower code $C^f_{2,i}$ is a $(k_2, d_2)$-locally recoverable code, with dimension at most $k_2=p-\rho_2 +1$ and minimum distance at least $d_2=\rho_2$.

    Denote now by $C^f_{1,i}$ the punctured code obtained from $C^f$ by retaining the codeword positions corresponding to the points in $F_{1,i}$, where we define
    \begin{equation*}
        F_{1,i}\coloneqq \{P\in T \mid \pi_z(P)=\gamma\}.
    \end{equation*}
    Note that these points are precisely the $\F_q$-points on $\mathcal{Z}_{\gamma},$ from which it follows that $n_1 \geq \eta p$. 

    The fact that the dimension of $C^f_{1,i}$ is at most $k_1=(\eta-\rho_1+1)(p-\rho_2+1)$ follows directly by definition of the space of polynomials $V$ in \cref{def:main:artinschreier}. On the other hand, to show that $C^f_{1,i}$ has minimum distance at least $d_1=\rho_1\rho_2$, we start by computing a lower bound on the number of coordinates that uniquely determine a vector in the code. Let
    \begin{equation}
    \label{eq:gpolynomial}
        g_\gamma(x,y)\coloneqq  g(x,y, \gamma) = g_0(x)+g_1(x)y+\dots g_{p-\rho_2}(x)y^{p-\rho_2},
    \end{equation}
    where $g_j\in\F_q[x]$ with $\deg(g_j)\leq \eta-\rho_1$ for each $0\leq j\leq p-\rho_2$. 
    Considering now the process of polynomial interpolation in the setting of this multivariate polynomial $g$, we obtain the following.
    
    For $a\in \F_q$ we define the following subset of $T$:
    \begin{equation*}
        W_a\coloneqq \{(a,y)\in \Zgamma(\F_q) : y^p - y = f(a, \gamma)\}.
    \end{equation*}
    Note that, by definition of the evaluation set $T$ and of the set $\Gamma$, there exist at least $\eta$ distinct nonempty sets $W_a$, since for any point in $T$, and hence in particular for any point in $F_{1,i}$, it holds that the $z$-coordinate is an element of the set $\Gamma$. Moreover, if $W_a\neq \emptyset$ then $|W_a|=p$ by Artin-Schreier theory.

     As in the proof of \cite[Theorem 3.1]{berg2024codes}, the main idea is now to find each $g_j(x)$ in \cref{eq:gpolynomial} by Lagrange univariate polynomial interpolation, which requires finding $g_j(a)$ for at least $\eta-(\rho_1-1)$ values of $a$ with $W_a\neq \emptyset$. 
    To this aim, observe that
    \begin{equation*}
        g_\gamma(a,y)=g_0(a)+g_1(a)y+\dots g_{p-\rho_2}(a)y^{p-\rho_2}
    \end{equation*}
    is a univariate polynomial, and that its coefficients $g_j(a)$ can hence be interpolated given the values of $g_\gamma(a,y)$ on at least $p-(\rho_2-1)$ distinct points in $W_{a}$.

    For this layered process of interpolation on $g_\gamma(x,y)$ to succeed, we need $\eta-(\rho_1-1)$ sets $W_a\neq \emptyset$ for which the evaluations of $g_\gamma(a,y)$ are known on at least $p - (\rho_2-1)$ points in $W_a$.

    Note that, for a fixed $W_a\neq \emptyset$, determining $g_j(a)$ for all $j$ fails only if there are at least $\rho_2$ erasures at indices corresponding to points in $W_a$. Moreover, the next step of determining the coefficients of $g_j(x)$ for each $j$ fails only if there are at least $\rho_1$ values of $a\in \F_q$ with $W_a\neq \emptyset$ for which we failed to retrieve $g_j(a)$. 
    Therefore, there must be at least $\rho_1\rho_2$ missing evaluations of $g_\gamma(x,y)$ in $F_{1,i}$ for the recovery to fail, which in turn guarantees that the minimum distance of the middle code $C^f_{1,i}$ is at least $\rho_1\rho_2$. As observed in \cref{rem:mindistbounds:comparison}, from \cref{cor:vanishing}, we also have the lower bound $\rho_1p - s(p-\rho_2)$ for the minimum distance of $C^f_{1,i}$, from which $d_1 = \max\{\rho_1\rho_2, \ \rho_1 p - s(p-\rho_2)\}$ follows.
    
    To conclude the proof, we consider now the full code $C^f$ and show that the evaluation map is injective for $V$ and $T$, thus proving the dimension of $C$ to be exactly the dimension of $V$ as an $\F_q$-vector space. Moreover, combining the previous part of the proof with \cref{cor:mindist}, we also show that the bound $d\geq \rho_3\max\{\rho_1\rho_2, \ \rho_1 p - s(p-\rho_2)\}$ holds.

    To show the injectivity of the evaluation map for $V$ and $T$, let $g\in V\setminus\{0\}$. As noted in the proof of \cref{cor:mindist}, there must be at least $\rho_3$ values of $\gamma\in\Gamma$ such that $g(x,y,\gamma)$ is not identically zero on $\Zgamma$. Moreover, on each $\Zgamma$ for which $g$ is not identically zero, \cref{cor:vanishing} guarantees that $g$ vanishes in at most $Ms - (s-p)(\eta-\rho_1)$ points of $\ZZ$ aside from $P_\gamma$, and hence in particular it vanishes in at most $\rho_3(Ms - (s-p)(\eta-\rho_1))$ points in $T$. Since $|T|\geq \rho_3 \eta p > \rho_3(Ms - (s-p)(\eta-\rho_1))$ by condition $(2)$ in \cref{def:main:artinschreier}, it hence follows that any nonzero function in $V$ cannot vanish at all the points of $T$ and therefore, in particular, that the evaluation map on the points in $T$ is injective.

    To conclude, observe that the fact that there are at least $\rho_3$ values of $\gamma\in\Gamma$ such that any $g\in V\setminus\{0\}$ is not identically zero on $\Zgamma$, together with the lower bound $\rho_1 \rho_2$ for the minimum distance of $C^f_{1,i}$, implies in particular that $\rho_1 \rho_2 \rho_3$ is a lower bound for $d$. On the other hand, from \cref{cor:mindist} we already know that $\rho_3(\rho_1 p - s(p-\rho_2)$ is also a lower bound for $d$, hence the bound $d\geq \rho_3\max\{\rho_1\rho_2, \ \rho_1 p - s(p-\rho_2)\}$ follows. By \cref{rem:mindistbounds:comparison}, we also already know that $\max\{\rho_1\rho_2, \ \rho_1 p - s(p-\rho_2)\} = \rho_1\rho_2$ when $\rho_1 < s$, while $\max\{\rho_1\rho_2, \ \rho_1 p - s(p-\rho_2)\} = \rho_1 p - s(p-\rho_2)$ when $\rho_1 \geq s$.
\end{proof}

\begin{remark}
    \label{rem:explicitrecovery}
    As it was the case in the proof of \cite[Theorem 3.1]{berg2024codes}, observe that the proof of \cref{thm:AS:parameters} provides explicit recovery procedures both at the level of lower and at the level of middle codes.
\end{remark}

\section{Examples of HLRCs from Artin-Schreier surfaces}
\label{sec:ASexample}

In this section, we revisit examples of HLRCs from an Artin-Schreier surface discussed in \cite[Section 4]{berg2024codes}, showing how, with our approach, we are able to obtain improved bounds on their minimum distance and improved estimates on their asymptotic parameters.

Throughout this section, we let $p$ be an odd prime. With notations as in \cref{sec:AS}, the Artin-Schreier surface of the example treated in \cite[Section 4]{berg2024codes} is 
\begin{equation}
\label{eq:ASex:surface}
    \calS_f : y^p -y = x^{p+1}z^2+x^2z^{p+1}
\end{equation}
over $\F_{p^2}$, with fibers
\begin{equation}\label{eq:ASex:fibers}
    \Zgamma: y^p-y=x^{p+1}\gamma^2+x^2\gamma^{p+1},
\end{equation}
where $\gamma\in\F_{p^2}^*$.  

We consider any HLRC $C^f$ from $S_f$, according to \cref{def:main:artinschreier}, with $\eta=2p-1$.
Observe that this gives codes that still satisfy the hypotheses of \cite[Theorem 3.1]{berg2024codes}, while being different than the examples described in \cite[Theorem 4.4]{berg2024codes}, where $\eta$ was set to be equal to $p$. In particular, choosing $\eta=2p-1$ produces codes $C^f$, for any suitable choice of $\rho_1, \rho_2, \rho_3$, that are shorter than the codes considered in \cite[Theorem 4.4]{berg2024codes}. Indeed, by \cite[Lemma 4.3]{berg2024codes}, this means that the cardinality of the evaluation set $T$ is $|T|=\eta p |\Gamma|$, which is strictly smaller than the cardinality of the evaluation set of the codes defined in \cite[Theorem 4.4]{berg2024codes}. This is because the evaluation set of such codes contained also all the $\F_{p^2}$-rational points on the fibers $\Zgamma$ with $\gamma^{p-1} = \gamma^{1-p}$, which in our construction with $\eta=2p-1$ are excluded. 

With this choice of evaluation points in place, we now choose parameters $\rho_1, \rho_2, \rho_3$ such that $p+1 \leq \rho_1 \leq \eta  , \ 2 \leq \rho_2 \leq p, \ 1\leq \rho_3 \leq |\Gamma|.$
In this way, given any choice of $\rho_1, \rho_2, \rho_3$ as above, the space of evaluation polynomials of the corresponding $C^f$ is 
   \begin{equation} \label{eqn: exV}
        V := V_{\rho_1,\rho_2, \rho_3} :=\langle x^iy^jz^k: 0\leq i\leq 2p-1-\rho_1, 0\leq j\leq p-\rho_2, 0\leq k \leq |\Gamma|-\rho_3\rangle.
    \end{equation}
Observe that, in particular, the indices $\eta, \rho_1, \rho_2, \rho_3$ as above satisfy conditions $(1)$ and $(2)$ in \cref{def:main:artinschreier}.

Despite choosing a smaller evaluation set of points, this construction of codes $C^f$ following our \cref{def:main:artinschreier} still satisfies also the hypotheses of \cite[Theorem 3.1]{berg2024codes} (see the discussion in the proof of \cite[Theorem 4.4]{berg2024codes}, in particular after \cite[Eq. (4.12)]{berg2024codes}). Therefore, it is immediate to see that the hierarchical structure of $C^f$ and the local parameters of lower and middle codes are essentially the same as those described in \cite[Theorem 4.4]{berg2024codes}. On the other hand, with our choices, the length of the full code is shorter, but we manage to give a bound on the minimum distance of the middle and full code that improves the ones presented in \cite[Theorem 4.4]{berg2024codes}. 

Indeed, from \cref{rem:mindistbounds:comparison} it follows that for $\rho_1 = p+1$ we retrieve the same bounds as in \cite[Theorem 4.4]{berg2024codes}, while for $\rho_1 > p+1$ we obtain a larger lower bound, namely $\rho_1 p - (p+1)(p-\rho_2)$ for the minimum distance of the middle codes, and  $\rho_3(\rho_1 p - (p+1)(p-\rho_2))$ for the minimum distance of $C^f$.

\begin{remark}
\label{rem:ASex:comparison:BMW}
    By the hypotheses on $\eta, \rho_1, \rho_2$, it follows in particular that $\eta p - M(p+1) + (\eta-\rho_1) \geq 2p + 2,$
    hence, from \cref{cor:mindist}, we have that the minimum distance $d$ of $C^f$ satisfies $d\geq \rho_3(2p+2)$.
    On the other hand, the general lower bound for $d$ obtained from \cite[Theorem 4.4]{berg2024codes} is $d \geq p + 4$, and $\rho_3(2p+2) \geq p + 4$ for any $p\geq 2$ and any choice of $\rho_3$ as in \cref{def:main:artinschreier}.
\end{remark}

We conclude the section by considering several examples of codes $C^f$ for explicit choices of indices $\rho_1$, $\rho_2$ and $\rho_3$, discussing how these choices impact the minimum distance, the information rate and the relative distance of the codes.

\begin{example}
\label{ex:ASmaxdim}
    The code $C^f$ with maximum dimension is the one with minimal indices, namely, $\rho_1=p+1$ $\rho_2=2$ and $\rho_3=1$. Clearly, in this case, we have
    an $[n,k,d]$-code with
    \begin{align*}
        n&=(2p^2-p)(p-1)^2, & k_{\max}&=(p-1)^4, & d\geq 2p+2.
    \end{align*}
    Notice that the information rate is on the order of $1/2$.

    When taking $\rho_1=p+1$, $\rho_2=2$ and $\rho_3=p+1$, we have an $[n,k.d]$-code with parameters
    \begin{align*}
        n&=(2p^2-p)(p-1)^2, & k&=(p^2-3p+1)(p-1)^2, & d\geq (p+1)(2p+2),
    \end{align*}
    which also has information rate on the order of $1/2$. Note that this improves the minimum distance from Example \ref{ex:ASmaxdim} by a factor of $p+1$ while preserving the order of the rate.

It is also desirable to have both rate and relative distance bounded away from 0 as $p$ increases. A choice of parameters for which this is achieved is, for instance, $\rho_1=p+1$, $\rho_2=\frac{p-1}{2}$ and $\rho_3=\frac{(p-1)^2}{2}$, which give an $[n,k,d]$-code with
    \begin{align*}
        n&=(2p^2-p)(p-1)^2, & k&=\left(p+1\right)\left(\frac{p+3}{2}\right)\left(\frac{p^2+3}{2}-p\right), & d\geq \frac{(p-1)^2(p^2-1)}{4}.
    \end{align*}
    Observe that the asymptotic parameters are on the order of $1/8$ and $1/2$, respectively, for information rate and relative minimum distance.
\end{example}

\section{Codes on Kummer-type surfaces}
\label{sec:Kummer}

In this section, in the same spirit as in \cref{sec:AS}, we devise a general construction of HLRCs as AG codes from surfaces admitting a fibration where the fibers are Kummer curves.

As before, let $q\coloneqq p^h$ be a positive power of some prime $p$ and let $\lambda \in \mathbb{Z}_{>0}$ with $\gcd(p,\lambda)=1$. Furthermore, assume that $\Fq$ contains a primitive $\lambda$-th root of unity. Consider the surface in $\mathbb{A}^3_{(x,y,z)}$ defined by
\begin{equation}
\label{eq:genKummer}
    \calS_f : y^\lambda-f(x,z) =0,
\end{equation}
where 
\begin{equation}
\label{eq:K:fpoldef}
    f(x,z) := c \cdot x^{h} \cdot z^{\nu} \cdot \prod_{i=1}^{\mu-{h}}(x-a_iz)\in \mathbb{F}_q[x,z],
\end{equation}
with $c\in \mathbb{F}_q^*$, $\mu \in \mathbb{Z}_{>0}$, $h, \nu\in \mathbb{Z}_{\geq 0}$, $0\leq h \leq \mu-1$, $\mu, \nu < \lambda$, and $a_i\in \mathbb{F}_q^*$ with $a_i\neq a_j$ for any $i\neq j$. Let now $\gamma\in \mathbb{F}_q^*$ and define
\begin{equation}
\label{eq:genKummer:Zgamma}
    \Zgamma\coloneqq \begin{cases}
        y^\lambda-f(x,z)=0\\
        z-\gamma=0
    \end{cases},
\end{equation}
that is, the curve obtained by intersecting $\calS_f$ with the plane $z=\gamma$ in $\mathbb{A}^3_{(x,y,z)}$. Let now $\A^2_\gamma$ be as defined in \cref{subsec:agtools}. Since $\Zgamma$ lies on the plane $z-\gamma=0$, with slight abuse of notation, as in \cref{sub:ASgeneral}, we will also denote by $\Zgamma$ the plane curve in $\mathbb{A}^2_\gamma$ defined by $y^\lambda-f(x,\gamma)=0$. 
\begin{remark}
    \label{rem:Kummer:fibers}
    Note that, by our assumptions and by the fact that, in \cref{eq:K:fpoldef}, we choose $a_i\neq a_j$ for any $i\neq j$, it follows that the curve $\Zgamma$ is of Kummer type and such that the polynomial $f(x,\gamma)$ is a product of $\mu-h +1$ distinct factors, of which $\mu - h$ are not multiple of $x$. Indeed,
    \begin{equation*}
        f(x,\gamma) = c \cdot x^h \cdot \gamma^\nu \cdot \prod_{i=1}^{\mu-h}(x-a_i\gamma)
    \end{equation*}
    and $a_i\gamma \neq a_j \gamma$ for any $i\neq j$, since $\gamma\in \mathbb{F}_q^*$ and multiplication by $\gamma$ is an automorphism of the field $\Fq$.
\end{remark}

Still with the notations introduced in \cref{subsec:agtools}, we  denote the projective closure of $\mathbb{A}^2_\gamma$ by $\mathbb{P}^2_\gamma$, with homogeneous coordinates $(x:y:w)$, and let
\begin{equation}
\label{eq:Zgamma:proj:closure}
    \overline{\calZ}_\gamma:y^\lambda - F_\gamma(x,w) = 0,
\end{equation}
where $F_\gamma(x,w)$ is the homogenization of $f(x,\gamma)\in \mathbb{F}_q[x]$ with respect to a variable $w$. Since we are assuming $\mu < \lambda$, the polynomial $F_\gamma(x,w)$ is homogeneous of degree $\lambda$ that is a multiple of $w$. More precisely, 
\begin{equation}
\label{eq:Fgamma}
    F_\gamma(x,w)=c\cdot w^{\lambda-\mu} \cdot x^h \cdot \gamma^\nu \cdot \prod_{i=1}^{\mu-h}(x - a_i\gamma \ w).
\end{equation}
The curve $\overline{\calZ}_\gamma$ defined in \cref{eq:Zgamma:proj:closure} is the projective closure of $\Zgamma$ in $\mathbb{P}^2_\gamma$. Observe that this curve has a unique point at infinity, namely $P_\gamma\coloneqq(1:0:0)$,  which corresponds to the point $[(1:0:0),(\gamma, 1)]$ in the partial compactification of $S_f$ in $\mathbb{P}^2_{(x:y:w)}\times \mathbb{A}^1_z$.
As in \cref{sub:ASgeneral}, we start the discussion with some observations on the multiplicity of the point $P_\gamma$ on $\ZZ$.

\begin{lemma}\label{lemma:genKummer:multiplicities}
The point $P_{\gamma}$ is a point of $\ZZ$ of multiplicity $\lambda-\mu$. Moreover, the line $w=0$ is the only tangent of $\ZZ$ at this point and intersects it with multiplicity $\lambda$, while the line $y=0$ intersects $\ZZ$ at $P_\gamma$ with multiplicity $\lambda-\mu$.
\end{lemma}
\begin{proof}
    Let $h(x,y,w) := y^\lambda - c\cdot w^{\lambda-\mu} \cdot x^h \cdot \gamma^\nu \cdot \prod_{i=1}^{\mu-h}(x - a_i\gamma \ w)$ be the homogeneous polynomial defining $\ZZ$ in $\mathbb{P}^2_\gamma$, see \cref{eq:Zgamma:proj:closure} and \cref{eq:Fgamma}. Since the point $P_\gamma$ is the origin in the affine chart of $\P^2_\gamma$ where $x=1$, to compute its multiplicity in $\ZZ$ we dehomogenize $h(x,y,w)$ with respect to $x$ and obtain $y^\lambda - c\cdot w^{\lambda-\mu} \cdot \gamma^\nu \cdot \prod_{i=1}^{\mu-h}(1 - a_i\gamma \ w)$. From here, the proof is entirely analogous to that of \cref{lemma:singularity}.
\end{proof}

\subsection{Code construction}

In the spirit of \cref{def:main:artinschreier}, let now $\eta\in \mathbb{Z}_{>0}$ be a positive integer with
\begin{equation}
\label{eq:genKummer:eta}
    0 \leq \eta \leq \mathrm{max} \ \{|\pi_x(\calZ_\gamma(\Fq))| \ : \ \gamma\in \Fq\},
\end{equation}
and let $\Gamma \subseteq \Fq$ denote the subset
\begin{equation}
\label{eq:genKummer:Gamma}
    \Gamma := \{\gamma\in\Fq \ : \ |\pi_x(\calZ_\gamma(\Fq))| \geq \eta\}.
\end{equation}
For any $\gamma \in \Gamma$, we define the following subset of $\F_{q}$-rational points of $\calS_f$:
\begin{equation}
\label{eq:genKummer:Tgamma}
    T_\gamma := \{(x,y,\gamma)\in \calS_f(\F_{q})\mid (x,y)\in \calZ_\gamma(\F_{q})\}\setminus \{(\bar{x},0,\gamma) \mid \bar{x}\in \F_{q}, \ f(\bar{x}, \gamma)= 0\}.
\end{equation}
In our notations, this subset can be identified with the subset of points of $\ZZ$ in $\P^2_\gamma$ given by 
\begin{equation}
\label{eq:genKummer:Tgamma:Pgamma}
    \ZZ(\F_{q^2})\setminus \left( \{ (1:0:0)\}\cup \{(\bar{x}:0:1) \mid \bar{x}\in \F_{q}, \ f(\bar{x},\gamma)= 0\}\right).
\end{equation}

\begin{remark}
\label{rem:genKummer:Tgamma:cardinality}
    Note that, by our choice of the set $\Gamma$, we have that $|T_\gamma| \geq \eta \lambda$ for any $\gamma\in \Gamma$. Indeed, by our definition of $f(x,z)$ (see \cref{eq:K:fpoldef}) and by Kummer theory, given any $\bar{x}$ which is the $x$-coordinate of a point $P:=(\bar{x}: \bar{y} : \bar{w})\in T_\gamma$, there are exactly $\lambda$ distinct solutions in $\F_{q^2}$ to the equation $y^\lambda - f(\bar{x},\gamma) = 0$, which means that there are exactly $\lambda$ many distinct points in $T_\gamma$ with $x$-coordinate equal to $\bar{x}$.  
\end{remark}

\begin{definition} 
\label{def:genKummer} 
    Let $\calS_f$ be the surface defined in \cref{eq:genKummer}. Let also $\Zgamma$ be as in \cref{eq:genKummer:Zgamma} and $\T_\gamma$ as in \cref{eq:genKummer:Tgamma}, with $\eta$ and $\Gamma$ as defined above in \cref{eq:genKummer:eta} and \cref{eq:genKummer:Gamma}, respectively. Define the set
    \begin{equation}
    \label{eq:genKummer:T}
    T:= \bigcup_{\gamma\in \Gamma}T_\gamma.
    \end{equation}    
    Furthermore, let $\rho_1\in[\eta]$, $\rho_2\in [\lambda]$ and $\rho_3\in[|\Gamma|]$ be positive integers such that $\rho_1\geq \lambda$ and $\rho_2\geq 2$.
    Define the $\F_{q}$-vector space
    \begin{equation}
        \label{eq:genKummer:V}
        V:=V_{\rho_1,\rho_2,\rho_3}:=\langle x^iy^jz^l \mid 0\leq i \leq \eta -\rho_1, \, 0\leq j \leq \lambda-\rho_2, \, 0 \leq l \leq |\Gamma| - \rho_3 \rangle.
    \end{equation}
We define an evaluation code (see \cref{defn:evalcode}) $C^{f} := C(T,V)$ with evaluation set $T$ and vector space of polynomials $V$.
\end{definition}

\begin{remark}
    \label{rem:genKummer:length}
    From \cref{rem:genKummer:Tgamma:cardinality}, it directly follows that a code as in \cref{def:genKummer} has length $n := |T| \geq \eta \lambda |\Gamma|$. Indeed, by \cref{rem:genKummer:Tgamma:cardinality} we have that $|T_\gamma| \geq \eta\lambda$ and any two distinct sets $T_{\gamma}\subseteq T$ are disjoint.
\end{remark}

\subsection{Dimension, locality and lower bounds for the minimum distance} 
\label{subsec:genKummer:parameters}

In this subsection, which is the counterpart of \cref{subsec:AS:parameters}, we discuss the parameters of a code $C^{f}$ as in \cref{def:genKummer}. We start by computing a bound on the minimum distance $d$ of $C^{f}$, which is derived by applying B\'ezout's theorem on each fiber $\ZZ$ and subtracting the minimum vanishing multiplicity of the functions in $V$ at $P_\gamma$, for all $\gamma\in \Gamma$. 

Let $M\coloneqq \eta + \lambda - \rho_1 - \rho_2$ and, similarly to \cref{eq:V:homog}, define the $\mathbb{F}_{q}$-vector space of (homogeneous) forms
\begin{equation}
\label{eq:V:homog:genKummer}
    V_H\coloneqq \langle x^iy^jw^\ell : 0\leq i\leq \eta-\rho_1, \ 0\leq j\leq \lambda-\rho_2,\ i+j+\ell=M\rangle.
\end{equation}
As noted in \cref{subsec:AS:parameters}, in the discussion following \cref{eq:V:homog}, any function $g(x,y,z)\in V$ restricts to $\Zgamma$ as a function $g_{\gamma}(x,y)=g(x,y,\gamma)$, and each such function can be written as $g_{\gamma}=G_\gamma(x,y,1)$ for some $G_\gamma\in V_{H}$, which is the homogenization of $g_\gamma$ with respect to some variable $w$.
The vanishing of $G_\gamma$ on affine points of $\ZZ$ with $w=1$ is the same as that of $g_{\gamma}$ on $\Zgamma$. Moreover, the vanishing multiplicity of $g\in V$ at $P_\gamma$ is also described by the vanishing multiplicity of $G_\gamma$ at $P_\gamma$, which is exactly the intersection multiplicity of the curve defined by $G_\gamma$ and $\ZZ$ at the origin, in the affine chart with $x=1$ of $\P^2_\gamma$.
We hence obtain the following lemma and corollary, which are the counterpart of \cref{lemma:vanishing:at:infty} and \cref{cor:vanishing}.

\begin{lemma}
\label{lemma:vanishing:at:infty:genKummer}
    Let $\overline{\calZ}_\gamma$ and $P_\gamma\coloneqq (1:0:0)$ be as defined above, and let $V_{H}$ be the $\mathbb{F}_{q}$-vector space of forms defined in \cref{eq:V:homog:genKummer}. Let $\mathcal{G}$ denote the curve in $\mathbb{P}^2_\gamma$ defined by the vanishing of the form $G \in V_{H}\setminus \{0\}$. Then 
    \begin{equation*} 
    \mathrm{min}_{G\in V_{H}\setminus \{0\}}\{i_{P_\gamma}(\mathcal{G}, \ZZ)\}=(\lambda-\mu)(\lambda-\rho_2). 
    \end{equation*}
\end{lemma}
\begin{proof}
    Let \begin{equation*} G(x,y,w)=\sum_{i=0}^{\eta-\rho_1}\sum_{j=0}^{\lambda-\rho_2}b_{i,j}x^iy^jw^{M-i-j}, \end{equation*} where $b_{i,j}\in \F_{q}$ and at least one $b_{i,j}\neq 0$.
   
    Observe that, by \cref{lemma:genKummer:multiplicities}, the degree-1 form $y$ vanishes with multiplicity $\lambda-\mu$ at the point $P_\gamma$ of $\overline{\calZ}_\gamma$ and the degree-1 form $w$ vanishes with multiplicity $\lambda$ at the point $P_\gamma$ of $\overline{\calZ}_\gamma$.  Thus if $b_{i,j}\neq 0$, the monomial $b_{i,j}x^iy^jw^{M-i-j}$ vanishes with multiplicity $j(\lambda-\mu)+(M-i-j)\lambda$ and the degree-$M$ form $G$ vanishes with multiplicity at least \[\min\{j(\lambda-\mu)+(M-i-j)\lambda:0\leq i\leq\eta-\rho_1, 0\leq j\leq \lambda-\rho_2\}\] at $P_{\gamma}$. This minimum is $(\lambda-\mu)(\lambda-\rho_2)$, since $\lambda-\mu < \lambda$. 
\end{proof}

\begin{corollary}\label{cor:vanishing:genKummer}
Every non-zero function in $V_H$ vanishes in at most $M\lambda-(\lambda-\mu)(\lambda-\rho_2)$ points of $\ZZ$ aside from $P_{\gamma}$.    
\end{corollary}
\begin{proof}
    Follows directly from B\'ezout's theorem.
\end{proof}

From this, we get the counterpart of \cref{cor:mindist} for a code $C^f$ as in \cref{def:genKummer}.

\begin{corollary}\label{cor:mindist:genKummer}
    The minimum distance of $C^f$ is at least 
        $\rho_3\left(\eta\lambda - M\lambda + (\lambda-\mu)(\lambda-\rho_2)\right) = \rho_3 (\lambda \rho_1 - \mu(\lambda - \rho_2)).$
\end{corollary}
\begin{proof}
Let $g(x,y,z)\in V$.  As there are $|\Gamma|$ distinct $z$-values for points of the evaluation set $T$, and the maximum degree of $g$ in $z$ is $|\Gamma| - \rho_3$, there must be at least $\rho_3$ values of $\gamma\in\Gamma$ such that $g(x,y,\gamma)$ is not identically 0.  For each of these values of $\gamma$, $g$ has the same vanishing on $\Zgamma$ as the corresponding $G_{\gamma}\in V_{H}$ on affine points of $\ZZ$. Corollary \ref{cor:vanishing:genKummer} implies that $G_{\gamma}$ vanishes on at most $M\lambda-(\lambda-\mu)(\lambda-\rho_2)$ affine points of $\ZZ$, thus on at most this number of points of $\Zgamma$.  Therefore there must be $\eta\lambda - M\lambda + (\lambda-\mu)(\lambda-\rho_2) = \lambda \rho_1 - \mu(\lambda - \rho_2)$ evaluation points of $\Zgamma$ where $g$ does not vanish, and the desired result follows.
\end{proof}

\begin{remark}
    When $\rho_2 = \lambda$, the minimum computed in \cref{lemma:vanishing:at:infty:genKummer} is zero, since there are functions in $V$ (and hence corresponding forms in $V_H$) that do not vanish at $P_\gamma$. This means that the bound on the minimum distance obtained in \cref{cor:mindist:genKummer}, which is still positive, is simply the estimate coming from Bézout’s Theorem on the fibers $\Zgamma$.
\end{remark}

\begin{remark} 
\label{rem:genKummer:d:indep:bound}
The lower bound given on the minimum distance $d$ of a code $C^f$ in \cref{cor:mindist:genKummer} is always positive. More precisely, we have
\begin{align*}
    \eta\lambda - M\lambda + (\lambda-\mu)(\lambda-\rho_2) &= \eta\lambda - (\eta+ \lambda -\rho_1-\rho_2)\lambda + (\lambda-\mu)(\lambda-\rho_2)\\
    &= -\lambda^2 + (\rho_1+\rho_2)\lambda + (\lambda-\mu)(\lambda-\rho_2)\\
    &=\mu\rho_2 + \rho_1\lambda - \mu\lambda 
\end{align*}
and, since $\mu\rho_2 \geq \mu$ because $\rho_2\geq 2$, and $\rho_1\lambda - \mu\lambda \geq 0$ because $\rho_1\geq \lambda$, we hence have 
\begin{equation}
\label{eq:d:indep:bound:genKummer}
    \eta\lambda - M\lambda + (\lambda - \mu)(\lambda-\rho_2) \geq \lambda^2 - \mu\lambda + \mu >0.
\end{equation}
Note that \cref{eq:d:indep:bound:genKummer} implies, in particular, that we always have at least the lower bound $\lambda^2 - \mu\lambda + \mu$ for the minimum distance of a code $C^f$ defined in \cref{def:genKummer}, regardless of the choice of the parameters $\rho_1,\rho_2,\rho_3$. 
\end{remark}

\begin{remark}
    With our choice of the sets $T_\gamma$ in \cref{eq:genKummer:Tgamma}, we do not include in the set of evaluation points $T$ all the points in $\{(\bar{x},0,\gamma) \mid \bar{x}\in \F_{q}, \ f(\bar{x},\gamma)= 0\}$, for any $\gamma\in \Gamma$.
    Observe that these are exactly the $\F_q$-rational ramification points of the projection map $\pi_{x,z}$, for which our recovery procedure would not work. Indeed, the preimage of $(\bar{x},\gamma)$ by $\pi_{x,z}$ has cardinality strictly less than $\lambda$, for any $\bar{x}\in\mathbb{F}_{q}$ with $f(\bar{x},\gamma)=0$. 
    
    With respect to the approach used in \cref{lemma:vanishing:at:infty:genKummer}, considering these points does not improve the bound on minimum distances. In fact, there could be functions in $V$ vanishing at no point in $\{(\bar{x},0,\gamma) \mid \bar{x}\in \F_{q}, \ f(\bar{x},\gamma)= 0\}$. In fact, the only cases in which all functions in $V$ would vanish at all those points are those in which $\rho_1=\eta$, $\rho_3=|\Gamma|$ and $\rho_2\leq \lambda -1$, which are not particularly interesting since then the space of evaluation functions would simply be 
    $V:=\langle y^j \mid 0\leq j\leq \lambda - \rho_2  \rangle$.
\end{remark}

\cref{cor:mindist:genKummer} yields the lower bound $d \geq \rho_3\left(\eta\lambda - M\lambda + (\lambda-\mu)(\lambda-\rho_2)\right)$ for a code $C^f$ as in \cref{def:genKummer}. In the following theorem, which is the counterpart of \cref{thm:AS:parameters}, we show that any such code is an HLRC with lower codes supported on the fibers of $\pi_{x,z} \colon \calS_f \to \mathbb{A}^2_{x,z}$ and middle codes supported on the curves $\Zgamma$, see \cref{HLRCdef} and \cref{rem:middle:lower:codes}).

\begin{theorem}
\label{thm:genKummer:main}
    Let $\mathcal{S}_f$ and $\Zgamma$ be as in \cref{eq:genKummer} and \cref{eq:genKummer:Zgamma}, respectively. Let $C^f$ be an evaluation code from $\mathcal{S}_f$ as in \cref{def:genKummer}, obtained with any suitable choice of the parameters $\eta, \rho_1, \rho_2, \rho_3$.
    
    The code $C^f$ has hierarchical recovery with lower codes supported on the fibers of \linebreak $\pi_{x,z} \colon \calS_f \to \mathbb{A}^2_{x,z}$ and middle codes supported on the curves $\Zgamma$. More precisely, $C^f$ is an $[n,k,d]$-code with
    \begin{equation*}
        n=|T|\geq \eta\lambda |\Gamma|, \quad k=(\eta-\rho_1+1)(\lambda - \rho_2 + 1)(|\Gamma|-\rho_3 + 1), \quad d\geq \rho_3 \ \mathrm{max}\{\rho_1\rho_2, \   \lambda \rho_1 - \mu(\lambda - \rho_2)\}.
    \end{equation*}
    Moreover, with notations as in \cref{HLRCdef}, $C^f$ has local parameters
    \begin{equation*}
        ((n_1,k_1,d_1),(n_2,k_2,d_2)),
    \end{equation*}
    where
    \begin{align*}
        n_1&\geq \eta\lambda, &
        k_1&=(\eta-\rho_1+1)(\lambda - \rho_2 + 1), &
        d_1&= \mathrm{max}\{\rho_1\rho_2, \ \lambda \rho_1 - \mu(\lambda - \rho_2)\}, \\
        n_2&=\lambda, &
        k_2&=\lambda-\rho_2 +1, &
        d_2&= \rho_2.
    \end{align*}
\end{theorem}
\begin{proof}
    The proof of this theorem is, mutatis mutandis, entirely analogous to that of \cref{thm:AS:parameters}, so we omit it for conciseness.
\end{proof}

\section{Examples of HLRCs from Kummer-type surfaces}
\label{sec:KummerEx}

In this section, we present come explicit examples of HLRCs from surfaces with Kummer fibers. The inspiration for these examples stems from \cite{Vthesis} and \cite{BMV23}, where Weierstrass semigroups at all the points and the automorphism group of the curve $\calZ_1$ (see \cref{eq:ex:Kummer:Zgamma}) were studied. We refer to \cref{ss:ArtinSchreier&KummerCovers} for the notation and terminology on one-variable function fields used throughout this section.

Let $q\equiv 2 \pmod 3$ and $m\coloneqq (q+1)/3$. 
With notations as in \cref{sec:Kummer}, let 
\begin{equation*}
    f(x,z) := - z^mx^{2m} - x^mz^{2m} = - x^mz^m (x^m+z^m)
\end{equation*}
and consider the algebraic surface  
\begin{equation}
\label{eq:ex:Kummer} 
\mathcal{S}_f: y^{q+1} - f(x,z) =
y^{q+1}+z^mx^{2m}+x^mz^{2m}=0.
\end{equation}
over $\mathbb{F}_{q^2}$ in $\mathbb{A}^3_{(x,y,z)}$.

To construct evaluation codes on $\mathcal{S}_f$ as in \cref{def:genKummer}, we proceed as follows. To choose the set of evaluation points $T$, we consider the points in $\mathcal{S}_f(\mathbb{F}_{q^2})$ that lie on curves 
\begin{equation}
\label{eq:ex:Kummer:Zgamma}
\calZ_{\gamma}: y^{q+1}+\gamma^{m}x^{2m}+x^m \gamma^{2m}=0,
\end{equation}
for $\gamma \in \mathbb{F}_{q^2}\setminus\{0\}$.
With the notations introduced in \cref{subsec:agtools}, let $\overline{\calZ}_\gamma$ denote the projectivization of $\calZ_\gamma$ in $\P^2_\gamma$, and let $\calY_\gamma$ denote the normalization of $\overline{\calZ}_\gamma$.
In the following lemma, we count the number of $\F_{q^2}$-rational points of $\overline{\calZ}_\gamma$ and $\calY_\gamma$.

\begin{lemma}\label{lemma:pointCountKummer}
Let $\gamma \in \mathbb{F}_{q^2}\setminus\{0\}$.  
The curve $\overline{\calZ}_\gamma$ has $(q^3+2q^2+2q+7)/3$ $\mathbb{F}_{q^2}$-points and its normalization $\calY_\gamma$ has $(q^3+2q^2+4q+3)/3$ $\mathbb{F}_{q^2}$-points.
\end{lemma}
\begin{proof}
We claim that $\overline{\calZ}_\gamma$ is covered by the Hermitian curve via an $\F_{q^2}$-rational covering map. By \cite[Theorem 10.2]{HKT}, this implies that $\overline{\calZ}_\gamma$ is maximal.
This yields that $\calY_\gamma$ is also maximal over $\F_{q^2}$, since it is birationally equivalent to $\overline{\calZ}_\gamma$.
Consider $\mathcal{H}_{q}: U^{q+1}+V^{q+1}+W^{q+1}=0$. Being a model of the Hermitian curve, $\mathcal{H}_{q}$ is maximal over $\mathbb{F}_{q^2}$. Moreover, the map $\varphi_{\gamma}: \mathcal{H}_{q}\rightarrow \overline{\calZ}_\gamma$ defined by
\[
U\mapsto U^3=x, \, V\mapsto UVW=y, \, W \mapsto \frac{1}{\gamma}W^3=w 
\]
is a triple cover, which proves our initial claim.

We now wish to count the $\F_{q^2}$-rational points of $\overline{\calZ}_\gamma$ and $\calY_\gamma$. We start by computing the number of $\F_{q^2}$-rational points of $\calY_\gamma$ with the Hasse-Weil Theorem and then deduce the cardinality of $\overline{\calZ}_\gamma(\F_{q^2})$. 
To this aim, we first compute the (geometric) genus $g(\calY_\gamma)$ of $\calY_\gamma$, that is, the genus of the function field $\F_{q^2}(\calY_\gamma)$. Note that $g(\calY_\gamma)$ is equal to the genus $g(\overline{\calZ}_\gamma)$ of $\overline{\calZ}_\gamma$, since the curves $\calY_\gamma$ and $\overline{\calZ}_\gamma$ are birationally equivalent, and hence their function fields $\F_{q^2}(\calY_\gamma)$ and $\F_{q^2}(\overline{\calZ}_\gamma)$ are isomorphic. See \cref{fig:genus} for a graphic summary of this discussion.

\begin{figure}[ht]
\adjustbox{center}{
\begin{tikzcd}
\mathcal{H}_q \arrow[d, "\varphi_\gamma"'] \arrow[rd, dotted]                        &                                       &  &                                                                                                                  & \mathbb{F}_{q^2}(\mathcal{H}_q) \arrow[ld, "3"', no head] \arrow[rd, "3", no head] &                                                                 \\
\overline{\mathcal{Z}}_\gamma \arrow[d, "\pi_x"'] \arrow[r, "\sim", no head, dashed] & \mathcal{Y}_\gamma \arrow[ld, dotted] &  & \mathbb{F}_{q^2}(\overline{\mathcal{Z}}_\gamma) \arrow[rd, "q+1"', no head] \arrow[rr, "\cong", no head, dashed] &                                                                                    & \mathbb{F}_{q^2}(\mathcal{Y}_\gamma) \arrow[ld, "q+1", no head] \\
\mathbb{P}^1_x                                                       &                                       &  &                                                                                                                  & \mathbb{F}_{q^2}(x)                                                                &                                                                
\end{tikzcd}
}
    \caption{The leftmost diagram summarizes the rational maps between the curves $\mathcal{H}_q$, $\overline{\mathcal{Z}}_\gamma$ and $\calY_\gamma$. The map $\pi_x$ is the projection onto the $x$-coordinate, the map $\varphi_\gamma$ is the triple cover defined above, and the horizontal dashed arrow represents the birational equivalence of $\overline{\mathcal{Z}}_\gamma$ and $\calY_\gamma$. The dotted arrows represent induced rational maps in the commutative diagram.
    The rightmost diagram is the \enquote{translation}, in the function field setting, of the leftmost diagram.  \label{fig:genus}}
\end{figure}
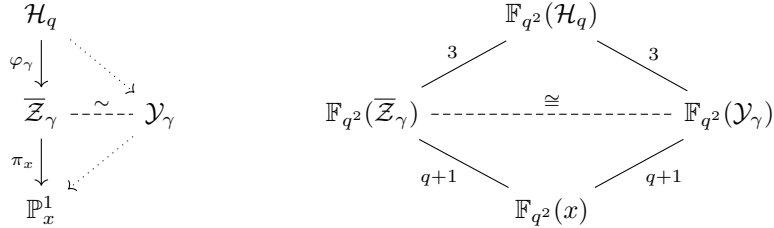

To compute the genus $g(\overline{\calZ}_\gamma)$ (that is, as observed, the genus of the function field $\F_{q^2}(\overline{\calZ}_\gamma)$), we apply the Riemann-Hurwitz formula (see \cite[Theorem 3.4.13]{Sti}) to the function field extension $\F_{q^2}(\overline{\calZ}_\gamma)/\F_{q^2}(x)$, which is a Kummer extension of degree $q+1$. By Kummer theory, it can be shown (see for instance \cite[Section 3.1]{Vthesis}) that the only ramified places in this extension are the zero $Q_0$ and pole $Q_\infty$ of $x$ and the $m$ distinct zeros $Q_1,\ldots, Q_m$ of $x^m + \gamma^m$. More precisely, both $Q_0$ and $Q_\infty$ have $m$ distinct extensions with ramification index $3$, while $Q_1,\ldots, Q_m$ are totally ramified. Hence, the Riemann-Hurwitz formula for $\F_{q^2}(\overline{\calZ}_\gamma)/\F_{q^2}(x)$ yields
\begin{align*}
    2(g(\overline{\calZ}_\gamma)-1) &= -2(q+1) + 2m + 2m + qm
    = \frac{q^2-q-2}{3},
\end{align*}
and hence we have $g(\calY_\gamma) = g(\overline{\calZ}_\gamma) = (q^2-q+4)/6.$

We can now compute the number of $\F_{q^2}$-rational points of $\calY_\gamma$, which by Hasse-Weil is
\begin{align*}
   |\calY_\gamma(\F_{q^2})| &= q^2 + 1 + 2 \cdot g(\calY_\gamma)\cdot q
   = q^2 + 1 + 2 \cdot \frac{q^2-q+4}{6}\cdot q\\
   &=\frac{q^3+2q^2+4q+3}{3}.
\end{align*}

To count the $\F_{q^2}$-rational points of $\overline{\calZ}_\gamma$, it now suffices to note that, except for the singular points $(1:0:0)$ and $(0:0:1)$, all the remaining $\F_{q^2}$-rational points of $\overline{\calZ}_\gamma$ are nonsingular and hence in one-to-one correspondence with $\F_{q^2}$-rational points of $\calY_\gamma$.

As already observed, from the discussion in \cite[Section 3.1]{Vthesis} it follows that there are exactly $m$ branches of $\overline{\calZ}_\gamma$ centered at the point $(0:0:1)$ and $m$ branches centered at $(1:0:0)$, whence we have $m$ points of $\calY_\gamma$ lying over $(0:0:1)$ and $m$ points of $\calY_\gamma$ lying over $(1:0:0)$. Indeed, the branches of $\overline{\calZ}_\gamma$ centered at $(0:0:1)$ (resp. $(1:0:0)$) are in one-to-one correspondence with the $m$ distinct places lying over the zero $Q_0$ (resp. the pole $Q_\infty$) of $x$ in $\F_{q^2}(\overline{\calZ}_\gamma)/\F_{q^2}(x)$. 
Therefore, the cardinality of $\overline{\calZ}_\gamma(\F_{q^2})$ is 
\[
    |\overline{\calZ}_\gamma(\F_{q^2})|= |\calY_\gamma(\F_{q^2})| - 2(m-1) = \frac{q^3+2q^2+2q+7}{3}.
\]
\end{proof}

To construct codes as in \cref{def:genKummer} from the surface $\mathcal{S}_f$ in this example, we consider the following sets $T_\gamma$ (see \cref{eq:genKummer:Tgamma}), for $\gamma\in \F_{q^2}\setminus\{0\}$,
\begin{equation*}
    T_\gamma \coloneqq  \{(x,y,\gamma)\in \calS(\F_{q^2}): (x,y)\in \calZ_\gamma(\F_{q^2})\}\setminus \left(\{(0,0,\gamma)\}\cup \{(a,0,\gamma) : a^m + \gamma^m = 0\} \right).
\end{equation*}
Observe that each such set $T_\gamma$ can be identified with the subset of $\overline{\calZ}_\gamma$ given by 
\[
    \overline{\calZ}_\gamma(\F_{q^2})\setminus \left( \{ (0:0:1), (1:0:0)\}\cup \{(a:0:1) :  a^m + \gamma^m = 0\}\right),
\]
which is exactly the set described in \cref{eq:genKummer:Tgamma:Pgamma} in the particular case we are considering in this section. Moreover, by \cref{lemma:pointCountKummer}, we can actually compute the cardinality of the sets $T_\gamma$ in our example. Indeed, by our assumptions on $q$, the equation $x^m + \gamma^m = 0$ has always $m$ distinct solutions over $\F_{q^2}$, hence the set $\{(a:0:1) :  \ a^m + \gamma^m = 0\}$ contains $m$ distinct points of $\overline{\calZ}_\gamma$. By Kummer theory (see the discussion in the proof of \cref{lemma:pointCountKummer} and \cite[Section 3.1]{Vthesis}), it is possible to see that each of these points is the center of a totally ramified place in $\F_{q^2}(\overline{\calZ}_\gamma)/\F_{q^2}(x)$.
We can then compute the cardinality of $T_\gamma$ using \cref{lemma:pointCountKummer}:
\begin{equation*}
    |T_\gamma| = |\overline{\calZ}_\gamma(\F_{q^2})| - 2 - m = \frac{q^3+2q^2+q}{3} = \frac{q(q+1)^2}{3} = q(q+1)m.  
\end{equation*}

Given $a \neq 0$ in $\mathbb{F}_{q^2}$, with $a^m+\gamma^m\neq 0$,
there are $q+1$ values of $y$ such that
$y^{q+1}+a^m\gamma^m(a^m + \gamma^m)=0$. 
In other words, 
\[
    |\{x\in\mathbb{F}_{q^2}: (x,y,\gamma)\in T_\gamma, \ \exists \ y\in \F_{q^2}\}|= \frac{|T_\gamma|}{q+1}=mq.
\]
In the notations of \cref{def:genKummer}, this means that we are setting $\eta \coloneqq  mq$, and the corresponding set $\Gamma$ is $\Gamma = \F_{q^2}\setminus\{0\}$. Writing $|T_\gamma|=\eta(q+1)$, we hence have that, for any choice of the parameters $\rho_1, \rho_2, \rho_3$, the length of codes $C^f$ constructed from $\mathcal{S}_f$ is equal to $|T|=|T_\gamma| \cdot |\Gamma| = \eta(q+1)(q^2-1)$. 

We let $q\neq 2$ and consider codes from $\mathcal{S}_f$ as in \cref{def:genKummer}, with $\eta$, $\Gamma$ and $T$ as just discussed, and with parameters $\rho_1, \rho_2, \rho_3$ such that $q+1 \leq \rho_1 \leq \eta  , \ 2 \leq \rho_2 \leq q+1, \ 1\leq \rho_3 \leq q^2-1.$
In this way, given any choice of $\rho_1, \rho_2, \rho_3$ as above, the space of evaluation polynomials of the corresponding code $C^f$ is 
   \begin{equation} \label{eqn:exKummer:V}
        V := V_{\rho_1,\rho_2, \rho_3} :=\langle x^iy^jz^k: 0\leq i\leq \eta-\rho_1, 0\leq j\leq q+1-\rho_2, 0\leq k \leq q^2-1-\rho_3\rangle.
    \end{equation}
Note that we are assuming $q\neq 2$ only to ensure that $\eta \geq q+1$, which is needed to be able to pick $\rho_1$ in the range $q+1 \leq \rho_1 \leq \eta$. 

The parameters of these codes $C^f$ follow from \cref{thm:genKummer:main}. However, similarly to what was done in \cref{sec:ASexample}, we now discuss some examples for specific values of the parameters $\rho_1, \rho_2, \rho_3$ that give particularly interesting results.

\begin{example}\label{ex:sharp:dist}

    For the codes considered in this section, the bound in \cref{cor:mindist:genKummer} reads as $d\geq (3\rho_1+2\rho_2-6m)m\rho_3$. Letting $\rho_1=\eta$, $\rho_2=q+1$ and $\rho_3=q^2-2$, we hence have that $d\geq \eta(q+1)(q^2-2)$.

    Note that, with this choice of $\rho_1,\rho_2, \rho_3$, we obtain a code $C^f$ with dimension $k=2$, where the space $V$ of evaluation functions is simply $V=\langle 1, z\rangle$. This is arguably a poor choice of the space of evaluation functions, especially since $z$ is constant on the sets $T_\gamma$. However, this example shows that the lower bound for $d$ from \cref{cor:mindist:genKummer} is sharp. Indeed, consider the polynomial $1-z\in V$ and note that it vanishes at $(\alpha,\beta,\gamma)\in T$ if and only if $\gamma=1$, that is, at exactly the points in $T_1\subseteq T$. The weight of the codeword associated to the evaluation of $1-z$ is then $n-|T_1|=\eta(q+1)(q^2-1)-\eta(q+1) = \eta(q+1)(q^2-2) $, which is exactly the bound from \cref{cor:mindist:genKummer} in this case.
\end{example}

\begin{example}
    For maximum dimension, and thus maximum rate, one should instead \\ choose the smallest possible values for the indices $\rho_1=q+1$, $\rho_2=2$, $\rho_3=1$. This gives an $[n,k,d]$-code with
    \begin{align*}
        n&=\eta(q+1)(q^2-1), & k_{\max} &= (\eta-q)(q + 1)(q^2-1), & d&\geq (q+1)^2 - m(q-1),
    \end{align*}
    where the information rate is approximately $1$.

    Moreover, since $\rho_3=1$, in this case the bound for the minimum distance of the full code coincides with the one for the minimum distance of the middle codes and, in fact, the minimum distance of the middle codes coincides with the minimum distance of the full code.
    Indeed, let $P(x,y)\in V$ be a polynomial that gives a codeword $c_P$ of minimum weight on a fiber $\tilde{\gamma}$. Since $\rho_3=1$, then the maximum degree in $z$ of a polynomial in $V$ is $q^2-2$ and we can consider the polynomial $\tilde{P}(x,y,z)\coloneqq (z-\gamma_1)\cdot \cdots  \cdot (z-\gamma_{q^2-2})\cdot P(x,y)\in V$, where $\gamma_i \neq \gamma_j$ for any $i\neq j$ and $\gamma_i\neq \tilde{\gamma}$ for all $i$. Observe that such polynomial vanishes identically on $q^2-2$ fibers, that is, on all the fibers except $\calZ_{\tilde{\gamma}}$. Since we know that on $\calZ_{\tilde{\gamma}}$ the restriction of $\tilde{P}$ gives a codeword of minimum weight, then this means that the codeword associated to $\tilde{P}$ in the full code has in fact the same weight. 
\end{example}

\begin{example}
    There are also choices of $\rho_i$ for which both rate and relative minimum distance are bounded away from $0$ as $q$ increases.
    
        If $q$ is even, choose $\rho_1=mq/2$, $\rho_2=1$, $\rho_3=q^2/2$ to get an $[n,k,d]$-code with
        \begin{align*}
            n&=\eta(q+1)(q^2-1)\approx\frac{q^5}{3}, & k &= \frac{q^2}{2} \left(\frac{mq}{2} + 1\right)(q+1) \approx \frac{q^5}{12}, & d&\geq m(q^2-3q) \frac{q^2}{4}\approx \frac{q^5}{12}.
        \end{align*}
        
        If $q$ is odd, choose $\rho_1=m(q-1)/2$, $\rho_2=1$, $\rho_3=(q^2-1)/2$ to get an $[n,k,d]$-code with
        \begin{align*}
            n&=\eta(q+1)(q^2-1)\approx\frac{q^5}{3}, \quad k= \frac{q^2+1}{2}\left(\frac{3m^2}{2} + 1\right)(q+1) \approx \frac{q^5}{12},\\  d &\geq m(q^2-4q-1) \frac{q^2-1}{4} \approx \frac{q^5}{12}.
        \end{align*}
  
    For large $q$, this yields asymptotically a value of $1/4$ for both information rate and relative distance.
\end{example}

\section{Observations, Generalizations, and Next Steps}

The general approach that motivates this paper is creating an evaluation code on a surface in $\A^3$ with a fibration into curves $\Zgamma$, where each curve has some regularity of structure and comes equipped with a natural cover $\phi_{\gamma}:\Zgamma\rightarrow\A^1$, and the fibers $\phi_{\gamma}^{-1}(y)$ for many $y\in\A^1$ are of equal cardinality over our field of interest across curves $\Zgamma$.  A small recovery group for a position corresponding to point $P$ is the set of positions corresponding to points on the fiber $\phi_{\gamma}^{-1}(\phi_{\gamma}(P))$, and a larger recovery group for that position corresponds to the set of evaluation points on $\Zgamma$. Functions for the code must be constructed so that, when any function is restricted to the nested fibers, the dimension of the punctured code is small enough to allow recovery of the desired number of positions. This paper describes how projectivizing each curve $\mathcal{Z}_{\gamma}$ can yield improvements in minimum distance bounds and dimension of the corresponding curve, and how, in our examples, a line on the projective surface can yield a fibration with the desired properties. Surfaces with Artin-Schreier and Kummer fibrations are excellent starting examples of our general approach to obtaining hierarchical codes. However, this general method can be applied much more widely, and some observations can be made on what would yield good codes in this context. 

Our first observation is that minimum distance bounds and dimension are largest in this construction when curves $\Zgamma$ have many evaluation points (requiring many points) relative to the degree of the curve.  This is because if an evaluation function $f$ restricts to a function of degree $e$ on $\Zgamma$, and the degree of the curve $\Zgamma$ is $s$, the function has at most $se$ zeros on $\Zgamma$ (minus the number of intersections at infinity).  If the number of evaluation points on $\Zgamma$ is large relative to $s$, this allows us to use evaluation functions of larger degree while still guaranteeing non-zero values, yielding a code with larger dimension.  Alternatively, a larger gap between $se$ and the number of evaluation points on $\Zgamma$ yields a larger minimum distance bound.  

However, we must still consider how this interacts with the length $n$ of the code, and we see as follows that  a curve with many points and a larger genus can yield a good rate.  The Hasse-Weil bound in Theorem \ref{thm:HasseWeil} states that the number of (projective) points over $\F_q$ for a curve of genus $g$ is at most $q+1+2g\sqrt{q}$.  For a maximal curve $\Zgamma$ with $g\geq 2\sqrt{q}$, the last term is dominant, so there are on the order of $2g\sqrt{q}$ points on $\Zgamma$.  For a smooth plane curve $C$ of degree $s$, the Pl\"{u}cker formula yields $g=\frac{(s-1)(s-2)}{2}$. So we use the heuristic that the genus grows approximately as $s^2/2$. In this regime we have around $s^2\sqrt{q}$ points on $\Zgamma$. Thus we need $se\leq s^2\sqrt{q}$. If we are seeking maximal rate, we would take $e$ as large as possible, so on the order of  $s\sqrt{q}$.  We assume that $\Zgamma\in\A^3_{x,y}$ and that the projection map is of the form $\phi_{\gamma}(x,y)=y$ and that this map has degree $t$, with $t\ll e$. This makes sense because the smallest recovery set had size $t$ and it is desirable for this set to be fairly small.  Then, the bound on the degree in $x$ for a function $f$ is on the order of $e$. With all of these assumptions, we then have that the rate of the full code is approximately $\frac{et|\Gamma|}{s^2\sqrt{q}|\Gamma|}=\frac{s\sqrt{q}t}{s^2\sqrt{q}}=\frac{t}{s}$.  When $t$ is close to $s$, we can achieve rates near 1.

A second observation is that if $C$ is a curve in $\A^2_{(x,y)}$, our framework can be easily applied to the cone over $C$ in $\A^3_{(x,y,z)}$.  This means that if $C$ is a curve with high genus, many points relative to its genus, and a cover of $\A^1$ with regular fibers over the desired field, our construction gives an easy way to construct a code with hierarchical recovery and good parameters.  The Hermitian curve is the maximal curve of maximal genus over $\mathbb{F}_{q^2}$, so this is a logical example to consider.
\begin{example}
    \label{rem:hermitian:cone}
    As in \cref{ex:hermitian}, let $\mathcal{H}_q$ be the Hermitian curve over $\F_{q^2}$ defined by the affine equation $y^q+y = x^{q+1}$. We consider the surface $\mathcal{S}_q$ in $\A^3$ defined by this same equation, i.e. the cone over the Hermitian curve.   

    Let \[T=\{(x,y,z): x,y,z\in\mathbb{F}_{q^2}, y^q+y=x^{q+1}\}\] be the set of evaluation points for the code, and let 
    \[V=V_{\rho_1,\rho_2,\rho_3}= \langle x^iy^jz^k:0\leq i\leq q^2-\rho_1, 0\leq j\leq q-\rho_2, 0\leq k\leq q^2-\rho_3\rangle,\] where $\rho_1\geq 2$, $\rho_2\geq 2$, $\rho_3\geq 1$, and $(q+1)\rho_2+q\rho_1>q^2+q$.

    For each $\gamma\in\mathbb{F}_{q^2}$, $\Zgamma$ is a copy of $\mathcal{H}_q$, $\Zgamma(\F_{q^2})=q^3$.  Note that the projectivization of $\Zgamma$ is tangent to the line at infinity and intersects it at the point $[0:1:0]$ with multiplicity $q+1$.  The map $\phi_{\gamma}:\Zgamma\rightarrow\A^1$ is given by $(x,y)\mapsto y$.  By arguments virtually identical to those in the Artin-Schreier case, we obtain hierarchical codes with length $n=q^5$, $k=(q^2-\rho_1+1)(q-\rho_2+1)(q^2-\rho_3+1)$, $d\geq\rho_3\max\{\rho_1\rho_2,\rho_1q-(q+1)(q-\rho_2)\}$.  The local parameters are also as in Theorem \ref{thm:AS:parameters}, with $q$ instead of $p$, $s=q+1$ and $\eta=q^2$.

    Note that we can choose $\rho_1=q$,  $\rho_2=2$, and $\rho_3=1$ to yield a code with rate on the order of 1.  
    \end{example}

\begin{remark}
     It is interesting also to consider the Hermitian curve as a Kummer curve and to discuss the parameters of codes $C^f$ obtained with our construction in \cref{def:genKummer} from the cone over the Hermitian curve $\mathcal{H}_q$ with map $\phi_{\gamma}(x,y)=x$. 

    To have the set $\Gamma$ as large as possible, we let $\eta:=q(q-1)$. Indeed, in this way $\Gamma = \F_{q^2}$, since for any $\gamma\in \F_{q^2}$ the set $T_\gamma$ as in \cref{eq:genKummer:Tgamma} on the fiber $\mathcal{H}_q \times \{z=\gamma\}$ consists of exactly all the affine rational points of $\mathcal{S}_q$ for which the projection map onto the $x$-coordinate is unramified. This means in particular that $|T_\gamma| = q^3-q = q(q-1)(q+1)$.

    Moreover, in the notations of \cref{eq:genKummer}, we have $\lambda = q+1$ and $\mu = q$. Choosing parameters $\rho_1= \lambda = q+1$, $\rho_2 = 2$, $\rho_3 = 1$, which is the choice yielding the largest bound on the minimum distance, we obtain a code $C^f$ with length $n = q^2(q^3-q)$, dimension $k = (q^2-2q)(q-1)(q^2-1)$, and minimum distance $d \geq (q+1)^2 - q(q-1) = 3q + 1$.

    Note that this code is longer than the code with same parameters $\rho_1, \rho_2, \rho_3$ obtained in \cref{sec:KummerEx}. Indeed, even though the equation of the surface considered in the examples of \cref{sec:KummerEx} has same $\lambda$ and $\mu$ as in this case, each set $T_\gamma$ has a smaller cardinality, and the set $\Gamma$ has a smaller cardinality as well. Therefore, it is not immediate to compare codes obtained from the Hermitian cone considered here and codes from \cref{sec:KummerEx}.

\end{remark}       

The cone construction described here has the potential to yield good hierarchical codes, but lacks the flexibility of the more general surface framework.  For example, in the surfaces of our example families in Sections \ref{sec:ASexample} and \ref{sec:KummerEx} are symmetrical with respect to $x$ and $z$, so admit two orthogonal fibrations into curves which both partition the fibers of constant $y$.  These curves could yield hierarchical codes with multiple choices of middle code, i.e. codes with local recovery and availability at the middle level.  Exploring the possibilities of recovery structures obtainable from these and related algebraic geometric codes is fertile ground for future work.

\section*{Acknowledgments}
This work stems from the project \emph{\enquote{Locally Recoverable Codes from higher-dimensional varieties}} led by B. Malmskog, C. Salgado and L. Vicino during the workshop \emph{Diversity in Finite Fields and Coding Theory} hosted at IMPA, Rio de Janeiro (Brazil), July 7-12 2024. It is our pleasure to thank IMPA for the hospitality.

C. Araujo was partially supported by CNPq and FAPERJ.

L. Costa was partially supported by CAPES.

B. Malmskog was partially supported by National Science Foundation award number DMS-2137661.

C. Salgado and L. Vicino were partially supported by an NWO Open Competition ENW – XL grant (project \enquote{Rational points: new dimensions} - OCENW.XL21.XL21.011).

Part of this work was done while B. Malmskog visited C. Salgado at the IAS, in Princeton. They thank the institute for their hospitality. During her stay at the IAS, C. Salgado was partially supported by both the James D. Wolfensohn Fund and the National Science Foundation under Grant No. DMS-1926686.

\end{document}